\documentclass[11pt]{amsart}
\usepackage[utf8]{inputenc}
\usepackage{amscd} 
\usepackage{lineno}
\usepackage[colorlinks,anchorcolor=blue,citecolor=blue,linkcolor=blue,urlcolor=blue,bookmarksdepth=2]{hyperref}
\usepackage[margin=1in]{geometry}
\usepackage{amsfonts, amsthm, amssymb, amsmath, stmaryrd, color}
\usepackage{mathrsfs,array,mathtools}

\usepackage{tikz}
\usetikzlibrary{cd,arrows}
\tikzset{>=stealth}
\tikzcdset{arrow style=tikz}
\tikzset{link/.style={column sep=1.8cm,row sep=0.16cm}}

\renewcommand{\O}{\mathrm{O}}
\newcommand{\vol}{\mathrm{vol}}
\newcommand{\Supp}{\mathrm{Supp}}
\newcommand{\Ml}{\mathrm{Ml}}

\newcommand{\End}{\mathrm{End}}
\newcommand{\Sh}{\mathrm{Sh}}
\newcommand{\bdelta}{\boldsymbol{\delta}}
\newcommand{\Eis}{\mathrm{Eis}}
\newcommand{\ab}{\mathrm{ab}}
\newcommand{\Tam}{\mathrm{Tam}}
\newcommand{\ord}{\mathrm{ord}}

\DeclareMathOperator{\Ad}{\mathrm{Ad}}

\newcommand{\bbA}{\mathbb{A}}

\newcommand{\bbC}{\mathbb{C}}

\newcommand{\bbQ}{\mathbb{Q}}

\newcommand{\bbR}{\mathbb{R}}

\newcommand{\bbZ}{\mathbb{Z}}

\newcommand{\calB}{\mathcal{B}}

\newcommand{\calD}{\mathcal{D}}

\newcommand{\calE}{\mathcal{E}}

\newcommand{\calN}{\mathcal{N}}

\newcommand{\calO}{\mathcal{O}}

\newcommand{\calU}{\mathcal{U}}

\newcommand{\frakA}{\mathfrak{A}}

\newcommand{\frakd}{\mathfrak{d}}

\newcommand{\frakg}{\mathfrak{g}}

\newcommand{\frakH}{\mathfrak{H}}

\newcommand{\frakk}{\mathfrak{k}}

\newcommand{\frakN}{\mathfrak{N}}

\newcommand{\frakp}{\mathfrak{p}}

\newcommand{\frakQ}{\mathfrak{Q}}

\newcommand{\frakt}{\mathfrak{t}}

\newcommand{\bfc}{\mathbf{c}}

\newcommand{\bff}{\mathbf{f}}

\newcommand{\bfh}{\mathbf{h}}

\newcommand{\bfn}{\mathbf{n}}

\newcommand{\Hom}{\mathrm{Hom}}

\newcommand{\Ind}{\mathrm{Ind}}
\newcommand{\Lie}{\mathrm{Lie}}

\newcommand{\tr}{\mathrm{tr}}

\newcommand{\GL}{\mathrm{GL}}

\newcommand{\Sp}{\mathrm{Sp}}

\newcommand{\GSp}{\mathrm{GSp}}

\newcommand{\Mat}{\mathrm{Mat}}

\newcommand{\Sym}{\mathrm{Sym}}

\newcommand{\diag}{\mathrm{diag}}
\newcommand{\sgn}{\mathrm{sgn}}

\newcommand{\fini}{\mathrm{fin}}
\newcommand{\bs}{\backslash}

\newcommand{\cusp}{\mathrm{cusp}}
\newcommand{\Aut}{\mathrm{Aut}}

\newcommand{\vep}{\varepsilon}
\newcommand{\isom}{\cong}
\newcommand{\abcd}{\begin{pmatrix}a & b \\ c & d \end{pmatrix}
}
\newcommand{\std}{\mathrm{std}}
\newcommand{\Gal}{\mathrm{Gal}}

\theoremstyle{plain}
\newtheorem{thm}{Theorem}[section]
\newtheorem{lem}[thm]{Lemma}
\newtheorem{cor}[thm]{Corollary}
\newtheorem{prop}[thm]{Proposition}

\theoremstyle{definition}

\newtheorem{rem}[thm]{Remark}

\newtheorem{defn}[thm]{Definition}

\title[Arithmeticity of critical $L$-values]{Arithmeticity of critical $L$-values for Hilbert-Siegel modular forms in the left half plane and $s=1$}
\author{Shuji Horinaga}
\address{NTT Institute for Fundamental Mathematics, NTT Inc, Japan}
\author{Hirotaka Kakuhama}
\address{Kyoto University, Japan}
\date{}

\begin{document}
\begin{abstract}
    For a cuspidal automorphic representation $\pi_f$ generated by a Hilbert-Siegel modular form $f$ and a Hecke character $\mu$ with parallel signature, 
    we show the arithmeticity of critical $L$-values of the standard $L$-function $L^S(s,\pi_f, \mu)$ twisted by $\mu$ in the left half plane.
    As a consequence, we show the arithmeticity at $s=1$ under the assumption that there exists a nonarchimedean place $v$ such that $\pi_v$ is unramified and $\mu_v$ is ramified.
\end{abstract}
\maketitle
% \vspace{-1em}
% \begin{center}
%   \begin{minipage}{0.9\textwidth}
%     \centering
%     {\small
%     $^{1}$ NTT Institute for Fundamental Mathematics, NTT Inc, Japan\\
%     $^{2}$ Kyoto University, Japan\\
%     }
%   \end{minipage}
% \end{center}

\section{Introduction}

One of the central themes in number theory is to investigate the arithmeticity of critical $L$-values for automorphic $L$-functions.
For elliptic modular forms, Shimura \cite{75_Shimura, 76_Shimura, 77_Shimura} and Manin \cite{75_Manin} establish the arithmeticity of such $L$-values.
Harris \cite{81_Harris} and Sturm \cite{81_Sturm} independently investigate the case of Siegel modular forms.
Garrett \cite{84_Garrett}, following the philosophy of the doubling integral, establishes the explicit formula for the restriction of Eisenstein series of $\Sp_{4n}$ and proves several arithmeticity results of Siegel modular forms.
In a series of papers \cite{83_Shimura_On_Eis_ser, 90_Shimura, 95_Shimura, 00_Shimura}, Shimura develops the theory of nearly holomorphic modular forms and proves the arithmeticity of critical values for several cases (see also \cite{85_boe,2000_Boe_Sch,81_Harris, 91_Mizumoto}).
Their works are basically restricted to the case of scalar valued Siegel modular forms.
The vector valued case has also been treated in \cite{HPSS, 00_Kozima, 2020_Liu, 95_Takayanagi} and \cite[Theorem 5.13]{22_Shih-yu_Chen}.

However, the arithmeticity at critical points in the left half plane, and at $s=1$ when $F=\bbQ$, remains open (cf.~\cite{2000_Boe_Sch}).
The central obstruction in the left half plane is the lack of results on local $\gamma$-factors and their arithmetic properties.
For the case $s=1$ and $F=\bbQ$, the associated Siegel Eisenstein series is not nearly holomorphic, making it difficult to discuss its behavior under the $\Aut(\bbC)$-action.
The present paper provides an Eisenstein series associated with a different section that is well-behaved under the $\Aut(\bbC)$-action.

Roughly speaking, to prove arithmeticity, we use the integral representation of $L^S(s, \pi, \mu)$
\[
L^S(s, \pi, \mu) \sim \langle E(g,s-1/2,f), \varphi\rangle
\]
where $\pi$ is a cuspidal automorphic representation of $\Sp_{2n}(\bbA_F)$ over a totally real field $F$, $\mu$ is a Hecke character of $F^\times F_{\infty,+}^\times \bs \bbA_F^\times$, $\varphi$ is an automorphic form in $\pi$, and $E(g,s,f)$ is the Eisenstein series on $\Sp_{4n}(\bbA_F)$ corresponding to a section $f_s$.
The rationality of the Eisenstein series implies the arithmeticity of the critical $L$-values.
In the left half plane, it is challenging to find the archimedean sections and calculate the corresponding zeta integrals.
In the case $s=1, F=\bbQ$ and $\mu^2 = \mathbf{1}$, Shimura's choice of sections does not give a nearly holomorphic Eisenstein series.
For these reasons, the study of the critical values at $s \leq 1$ has been limited.
When $\pi$ is unramified everywhere and $\mu = 1$, one may choose $f_s$ to be unramified everywhere, in which case one obtains the arithmeticity at $s=1$ only when $n \equiv 3 \bmod 4$ (see \cite{00_Kozima,91_Mizumoto}).

Shimura's strategy in the proof of arithmeticity is to define, at nonarchimedean places, the local sections $f^\Sh_{v,s}$ whose support is contained in the big cell.
This simplifies the computations of the zeta integrals, but we lose control of the Eisenstein series at the critical point $s=1$.
Our strategy is to introduce another section closely related to Shimura's.
The main idea of this paper is to consider a new local section defined by
\[
\bfh_{v,s} = M_{w,-s} f^\Sh_{v,-s}.
\]
The local zeta integrals associated with $\bfh_{v,s}$ coincide with the local $\gamma$-factors up to some explicit factors.
Furthermore, the corresponding Eisenstein series is nearly holomorphic even in the case $s=1$.
Combining the zeta integral of $\bfh_{v,s}$ with the functional equation for the standard $L$-functions \cite[Theorem 4]{2005_Lapid-Rallis}, we obtain the following arithmeticity result in the left half plane.
\begin{thm}[Corollary \ref{main_thm_s_0_negative}]\label{thm_main_intro}
    Let $F$ be a totally real field of degree $d$.
    Let $\pi$ be a cuspidal automorphic representation of $\Sp_{2n}(\bbA_F)$ and let $\mu$ be a Hecke character of $F^\times F_{\infty, +}^\times \bs \bbA_F^\times$.
    Let $S$ be a finite set of places containing the ramified places and the archimedean places.
    Suppose that for all archimedean places $v$, $\pi_v$ is a holomorphic discrete series representation of weight $\lambda_v = (\lambda_{v,1},\ldots,\lambda_{v,n})$, and $L(s,\pi,\mu)$ has a critical point $m$.
    Take $\varphi \in \pi$ such that the Fourier coefficients of the Hilbert-Siegel modular cusp form corresponding to $\varphi$ lie in a CM field.
    For any $\sigma \in \Aut(\bbC)$, we denote by $\pi_{^\sigma\varphi}$ the cuspidal automorphic representation generated by the twist ${^\sigma \varphi}$ of $\varphi$ by $\sigma$.
    When $m \leq 0$, we have 
    \begin{align*}
    &\sigma\left(\frac{L^S(m, \pi, \mu)}{D_F^{n/2}G(\mu)^n(2\pi\sqrt{-1})^{\sum_{v \mid \infty}\sum_{j=1}^n\lambda_{v,j} + dnm}\langle\varphi, \varphi\rangle}\right)\\
    &\qquad =
    \frac{L^S(m, \pi_{^\sigma\varphi}, {^\sigma\mu})}{D_F^{n/2}G({^\sigma\mu})^n(2\pi\sqrt{-1})^{\sum_{v \mid \infty}\sum_{j=1}^n\lambda_{v,j} + dnm}\langle{^\sigma\varphi}, {^\sigma\varphi}\rangle}
    \end{align*}
    except for the case where $m=0,F=\bbQ, S = S_\infty$ and $\mu = \mathbf{1}$.
    Here, $G(\mu)$ is the Gauss sum, $\langle\varphi,\varphi\rangle$ is the Petersson norm with respect to the Tamagawa measure, and $S_\infty$ is the set of archimedean places of $F$.
\end{thm}

We prove a more general version of the above theorem in Theorem \ref{thm_main}, which is formulated for two finite disjoint sets $S_1$ and $S_2$ of nonarchimedean places.
Theorem \ref{thm_main_intro} then follows by choosing $S_2 = \emptyset$.
We note that the case where $m=0,F=\bbQ, S = S_\infty$ and $\mu = \mathbf{1}$ is excluded from our main result.
This exception is particularly important, since it relates to the arithmeticity of the standard $L$-function for Siegel cusp forms at $s=1$ by the functional equation.
This case remains open when $n \equiv 1 \bmod 4$ and $n \neq 1$ by the functional equation.
See Remark \ref{rem_s=0_except} for details.

Another remaining case is the arithmeticity at $s=1, F=\bbQ$ and $\mu^2 = \mathbf{1}$.
As an application of our theorem, we obtain a partial arithmeticity result in this remaining case.
See also Remark \ref{rem_s=1} for a discussion of the condition in the following theorem.

\begin{thm}\label{thm_intro_s=1}
    With the notation of Theorem \ref{thm_main_intro}, suppose that $F=\bbQ,\mu^2 = \mathbf{1}$, and there exists a nonarchimedean place $v$ such that $\pi_v$ is unramified and $\mu_v$ is ramified.
    If $m=1$ is a critical point, then
    \begin{align*}
    &\sigma\left(
    \frac{L^S(1, \pi, \mu)}{D_F^{(n+1)/2}G(\mu)^{n+1}(2\pi \sqrt{-1})^{\sum_{v \mid \infty}{\sum_{j=1}^n \lambda_{v,j} + d(n+1)}}\langle\varphi, \varphi\rangle}\right)
    \\
    &\qquad
    =
    \frac{L^S(1, \pi_{{^\sigma \varphi}}, {^\sigma \mu})}{D_F^{(n+1)/2}G({^\sigma \mu})^{n+1}(2\pi \sqrt{-1})^{\sum_{v \mid \infty}{\sum_{j=1}^n \lambda_{v,j} + d(n+1)}}\langle{^\sigma \varphi}, {^\sigma \varphi}\rangle}.
\end{align*}
\end{thm}

We expect the same strategy to work for other groups, such as unitary groups.
More generally, our construction of sections may provide a way to study arithmeticity at critical points where the Eisenstein series arising from the standard choice of sections is not nearly holomorphic.
We also hope to pursue further number theoretic applications of our methods.

\subsection*{Acknowledgements}
The authors are grateful to Tamotsu Ikeda and Hidenori Katsurada for their valuable comments and encouragement. 
The authors would also like to thank Nobuki Takeda for sharing relevant references on the computations of zeta integrals.
The first named author is partially supported by JSPS KAKENHI Grant Number 23K12965 and JST CREST JPMJCR2113, Japan. The second named author is partially supported by JSPS KAKENHI Grant Number 23KJ0001.

The authors used ChatGPT (OpenAI) and Claude (Anthropic) during the preparation of this manuscript for language and expository editing.
All mathematical statements and proofs were independently verified by the authors, who take full responsibility for the content of this paper.

\section{Preliminaries}
%In this section, we fix some notation that will be used later.
\subsection{Basic notation and definitions}

    For a nonarchimedean local field $K$, let $\calO_K$ be the ring of integers of $K$ with maximal ideal $\frakp_K$, and let $q_K$ denote the order of the residue field $\calO_K/\frakp_K$.
    Fix a uniformizer $\varpi_K$ of $\calO_K$.
    We denote by $|\cdot|_K$ the absolute value normalized by $|\varpi_K|_K = q^{-1}_K$.
    Also, let $|\cdot|_K$ denote the absolute value when $K$ is archimedean.
    For a character $\chi$ of $K^\times$, let $c(\chi)$ be the conductor exponent of $\chi$.
    Throughout this paper, we fix an additive character $\psi_\bbQ = \otimes_v\psi_v$ of $\bbQ \bs \bbA$ such that $\psi_\infty(x) = \exp(2\pi\sqrt{-1}\,x)$ and $\psi_v$ is of order $0$ for any nonarchimedean place $v$.
    
    Let $F$ be a totally real field with $[F \colon \bbQ] = d$.
    We denote by $\bbA_F$ the adele ring of $F$ and by $\bbA_{F,\fini}$ the finite part.
    Let $F_\infty = \prod_{v \mid \infty} F_v$ and $F_{\infty,+}$ be the set of totally positive elements in $F_\infty$.
    For a nonarchimedean place $v$ of $F$, the symbols $\calO_v, \frakp_v, \varpi_v, q_v$, etc., denote the corresponding objects for $F_v$. 
    By $|\cdot|_F$, we mean the idele norm of $\bbA_F^\times$.
    Put $\psi_F = \psi_\bbQ \circ \tr_{F/\bbQ}$.
    The additive character $\psi_v$ stands for the $v$-component of $\psi_F$ for a place $v$ of $F$.
    Let $\mathbf{\bdelta}_v$ be an element of $F_v^\times$ such that $\bdelta_v\calO_{v}$ is the different $\frakd_v$ of $F_v$. 
    For $\lambda \in F_v^\times$, define $\psi_{v,\lambda}$ by $\psi_{v,\lambda}(x) = \psi_{v
    }(\lambda x)$.
    The additive character $\psi_{v, \bdelta_v^{-1}}$ is of order $0$.
    For $\theta \in F_v^\times$, let $\langle,\rangle_v$ denote the Hilbert symbol and $\chi_\theta(x) = \langle x,\theta \rangle_v$ be a quadratic character.
    By $\mathbf{1}$, we mean the trivial character.
    Throughout this paper, all representations are assumed to be admissible.
    Let $r$ be the right translation.
    For $B \in \Sym_n$, put $D_B = (-4)^{\lfloor n/2 \rfloor} \det(B)$.
    
    For a positive integer $n$, let $x_n$ be the scalar $n \times n$ matrix with the entry $x$.
    Let $R$ be a ring.
    By $\GL_n$ and $\Sp_{2n}$, we mean algebraic groups over $\bbZ$ whose $R$-rational points are given by
    \[
    \GL_n(R) = \{g \in \Mat_n(R) \mid \det(g) \in R^\times\}
    \]
    and
    \[
    \Sp_{2n}(R) = \left\{g \in \Mat_{2n}(R) \,\middle|\, {^tg}\begin{pmatrix}
        0_n & -1_n\\
        1_n & 0_n
    \end{pmatrix}g = \begin{pmatrix}
        0_n & -1_n\\
        1_n & 0_n
    \end{pmatrix}\right\},
    \]
    respectively.
    The symplectic group $\Sp_{2n}$ admits a Siegel parabolic subgroup $P_n$ defined by
    \[
    P_n = 
    \left\{
    g=
    \begin{pmatrix}
        A&B\\
        C&D
    \end{pmatrix}
    \in \Sp_{2n} \, \middle|\,
    C = 0_n
    \right\}.
    \]
    Let $P_n = M_nN_n$ be the Levi decomposition and
    \[
    \bfn(x) =
    \begin{pmatrix}
        1_n & x\\
        0_n & 1_n
    \end{pmatrix}
    \in N_n
    \]
    for $x \in \Sym_n$.
    We define a character $\Delta$ of $P_n$ by
    \[
    \Delta
    \left(
    \begin{pmatrix}
        A&B\\
        0_n&D
    \end{pmatrix}\right)
    =
    \det(A).
    \]
    Throughout this paper, we write $G_n = \Sp_{2n}$.

    For $G_n(\bbR)$, we fix the maximal compact subgroup $K_\infty$ by
    \[
    K_\infty = 
    \left\{
    \begin{pmatrix}
        A & B \\
        -B & A
    \end{pmatrix}
    \in 
    G_{n}(\bbR)
    \,\middle|\,
    {^tAA} + {^tBB} = 1_n
    \right\},
    \]
    which is isomorphic to $\mathrm{U}(n)$.
    For a nonarchimedean place $v$, put $K_v = G_{n}(\calO_{v})$. 
    We also use the following open compact subgroups $K_v(\frakp_{v}^m)$ of $G_n(F_v)$ defined by
    \[
    K_v(\frakp_{v}^m)
    =
    \left\{
    g \in G_n(\calO_{v})\mid g \equiv 1_{2n} \bmod \frakp_{v}^m\right\}
    \]
    for a nonnegative integer $m$.
    Put $K_\bbA = \prod_v K_v$.
    When the field $F$ is clear from the context, we denote the maximal compact subgroups $K_v$ and $K_\bbA$ by $K$ for short.

\subsection{Lie algebras, root systems, and root vectors}

    Let $\frakg_n$ be the complexified Lie algebra $\Lie(G_n(\bbR)) \otimes \bbC$ and $\frakk_n$ be the complexified Lie algebra of $K_\infty$ for an archimedean place $\infty$.
    The maximal compact subgroup $K_\infty$ defines a Cartan involution and the $(-1)$-eigenspace $\frakp_n$.
    The symmetric space $G_n(\bbR)/K_\infty$ is identified with the Siegel upper half space
    \[
    \frakH_n = \{z \in \Sym_n(\bbC) \mid \mathrm{Im}(z)>0\}.
    \]
    The complex structure of $\frakH_n$ yields the decomposition $\frakp_{n} = \frakp_{n,+} \oplus \frakp_{n,-}$.
    Then, $\frakp_{n,\pm}$ are commutative and isomorphic to $\Sym_n(\bbC)$ by the isomorphisms $\iota_\pm \colon \Sym_n(\bbC) \rightarrow \frakp_{n,\pm}$
    \[
    \iota_+(u) = \frac{1}{2} \beta \begin{pmatrix}
        0_n & \sqrt{-1}\,u\\
        0_n&0_n
    \end{pmatrix}\beta^{-1}, \
    \iota_-(u) = \frac{1}{2} \beta \begin{pmatrix}
        0_n & 0_n\\
        -\sqrt{-1}\cdot{^t u}&0_n
    \end{pmatrix}\beta^{-1}, \ \beta=
    \begin{pmatrix}
        -\sqrt{-1}\,1_n & \sqrt{-1}\,1_n\\
        1_n&1_n
    \end{pmatrix}.
    \]
    When the subscript $n$ is clear from the context, we write $\frakp_{n,\pm}$ as $\frakp_{\pm}$.
    For archimedean places $v$ of a totally real field $F$, let $\frakp_{\pm, v}$ be the corresponding Lie subalgebra of $\Lie(G_n(F_v)) \otimes \bbC$.

    Fix a fundamental Cartan subalgebra $\frakt_n$ contained in $\frakk_n$.
    The root systems are given by
    \[
    \Delta(\frakg_n, \frakt_n) = \{\pm(e_i - e_j) \mid 1 \leq i < j \leq n\} \cup \{\pm(e_i + e_j) \mid 1 \leq i \leq j \leq n\}
    \]
    and 
    \[
    \Delta(\frakk_n, \frakt_n) = \{\pm(e_i - e_j) \mid 1 \leq i < j \leq n\}, \qquad \Delta(\frakp_{n,\pm}, \frakt_n) = \{\pm(e_i + e_j) \mid 1 \leq i \leq j \leq n\}.
    \]
    We define the root vectors $X_{i,j}^\pm$ in $\frakp_{n,\pm}$ by $X_{i,j}^\pm = \iota_\pm(e_{i,j})$.
    Here, $e_{i,j}$ is the symmetric matrix whose $(i,j)$- and $(j,i)$-th entries are $1$ and whose other entries are $0$.

\subsection{Measures}

The Tamagawa measure $dx^\Tam$ on the adele ring $\bbA_F$ is the unique Haar measure that satisfies $\vol(F \bs \bbA_F) = 1$.
For a place $v$ of $F$, let $dx_v$ be the measure on $F_v$ so that $\vol(\calO_{v}) = 1$ when $v$ is nonarchimedean, and the Lebesgue measure when $v$ is archimedean.
The two measures are related by
\[
dx^\Tam = D_F^{-1/2} \prod_v dx_v.
\]
For $\Sym_n(F_v)$ (resp.\, $\Sym_n(F)\bs \Sym_n(\bbA_F)$), we define the measure by $\prod_{1 \leq i \leq j \leq n} dx_{v,i,j}$ (resp.\, the Tamagawa measure $d x^\Tam$), where $x_{v,i,j}$ is the usual coordinate and $dx_{v,i,j}$ is the measure on $F_v$ defined above.
We then have $dx^\Tam = D_F^{-n(n+1)/4} \cdot \prod_{v,1 \leq i \leq j \leq n} dx_{v,i,j}$.
Let $dg^\Tam$ be the Tamagawa measure on $G_{n}(\bbA_F)$.
For a nonarchimedean place $v$ of $F$, let $dg_v$ be the Haar measure so that $\vol(G_{n}(\calO_{v})) = 1$.
For an archimedean place $v$, let $dg_v$ be the measure defined in \cite{2021_Zheng_Liu_arch_zeta_int}, namely the product measure of the Haar measure on $K_\infty$ normalized so that $\vol(K_\infty) = 1$ and the measure on $G_n(\bbR)/K_\infty = \frakH_n$ defined by an invariant measure $\det((y_{i,j})_{i,j})^{-n-1} \prod_{1 \leq i \leq j \leq n} dx_{i,j}dy_{i,j}$.
Since the Tamagawa number in $G_{n}$ is $1$, the computation of \cite{61_Klingen} shows
\begin{equation}\label{Tamagawa_measure_comparison}
dg^\Tam = \frac{D_F^{-n(2n+1)/2}\pi^{dn(n+1)/2}}
{\prod_{j=1}^{n}\Gamma(j)^d\zeta_F(2j)} \prod_v dg_v,
\end{equation}
where $\zeta_F(s)$ is the Dedekind zeta function of $F$.
    
\section{Siegel modular forms and automorphic forms on \texorpdfstring{$G_n(\bbA_F)$}{}}\label{section_rationality_autom_form}
    %In this section, we review basic properties and the rationality of nearly holomorphic Hilbert-Siegel modular forms following \cite{HPSS, 00_Shimura}.
\subsection{Differential operators}
    We first recall the differential operators and the definition of nearly holomorphic functions on the Siegel upper half space $\frakH_n$.
    Let $(\rho,V)$ be an irreducible finite dimensional representation of $\GL_n(\bbC)$ and $T = \Sym_n(\bbC)$.
    For a nonnegative integer $e$, consider the space $\Ml_e(T,V)$ of $e$-multilinear forms from $T$ to $V$ and the subspace $S_e(T,V)$ consisting of elements in $\Ml_e(T,V)$ invariant under the permutations of indices.
    For $h \in \Ml_e(T,V), (u_1,\ldots,u_e) \in T^e$ and $a \in \GL_n(\bbC)$, define the representations $\rho \otimes \tau^e$ and $\rho \otimes \sigma^e$ on $\Ml_e(T,V)$ by
    \begin{align*}
        ((\rho \otimes \tau^e)(a)h)(u_1,\ldots,u_e)
        &=\rho(a)h({^t}au_1a,\ldots,{^t}au_ea),\\
        ((\rho \otimes \sigma^e)(a)h)(u_1,\ldots,u_e)
        &=\rho(a)h(a^{-1}u_1 \cdot {^t}a^{-1},\ldots,a^{-1}u_e \cdot{^t}a^{-1}).
    \end{align*}
    We will denote the restriction of $\rho \otimes \tau^e$ and $\rho \otimes \sigma^e$ to $S_e(T,V)$ by the same symbol.

    By a rational structure of $V$, we mean a $\bbQ$-vector subspace $V_\bbQ$ of $V$ such that $V_\bbQ \otimes \bbC = V$ and $V_\bbQ$ is stable under the action of $\GL_n(\bbQ)$.
    Fix rational structures $T_\bbQ = \Sym_n(\bbQ)$ on $T$ and $V_\bbQ$ on $V$.
    These rational structures induce rational structures on $\Ml_e(T,V)$ and $S_e(T,V)$.
    For explicit descriptions, see \cite[\S 3.3]{HPSS}.
    Take a basis $\{\vep_\nu\}_\nu$ of $T_\bbQ$.
    For $u \in T$, write $u = \sum_\nu u_\nu \vep_\nu$, and for $z \in \frakH_n$, write $z = \sum_\nu z_\nu \vep_\nu$.
    
    We define the differential operators $D,\overline{D},C,E \colon C^\infty(\frakH_n,V) \rightarrow C^\infty(\frakH_n,S_1(T,V))$ by
    \begin{align*}
        ((Df)(z))(u) = \sum_{\nu} u_\nu \frac{\partial f}{\partial z_\nu}(z), \qquad &((\overline{D}f)(z))(u) = \sum_{\nu} u_\nu \frac{\partial f}{\partial \overline{z_\nu}}(z)\\
        ((Cf)(z))(u) = 4((Df)(z))(yuy), \qquad &((Ef)(z))(u) = 4((\overline{D}f)(z))(yuy).
    \end{align*}
    Here, $u = (u_\nu)_\nu \in T, z = (z_\nu)_\nu \in \frakH_n$ and $y = \mathrm{Im}(z)$.
    We define $D_\rho^ef \in C^\infty(\frakH_n, S_e(T,V))$ by 
    \[
    D_\rho^ef(z) = (\rho \otimes \tau^e)(2y)^{-1} (C^e(\rho(2y)f))(z).
    \]

    Hilbert-Siegel modular forms are functions on $\prod_{v \mid \infty} \frakH_n$.
    By $D_v, \overline{D_v}, C_v, E_v$ and $D_{v,\rho_v}$, we mean the differential operators defined on the $v$-component of $\prod_{v \mid \infty} \frakH_n$.

\subsection{Nearly holomorphic Hilbert-Siegel modular forms}
    
    Let $(\rho, V)$ be a finite-dimensional representation of $\GL_n(\bbC)$.
    For an integer $e \in \bbZ_{\geq 0}$, let $N^e(\frakH_n)$ be the space of polynomials of degree at most $e$ in the entries of $y^{-1}$ with coefficients in holomorphic functions $\frakH_n \rightarrow \bbC$.
    Set $N_\rho^e(\frakH_n) = N^e(\frakH_n) \otimes V$.
    We define 
    \[
    N_\rho(\frakH_n) = \bigcup_e N_\rho^e(\frakH_n)
    \]
    to be the space of nearly holomorphic functions on $\frakH_n$.
    Equivalently, $f$ lies in $N_\rho^e(\frakH_n)$ if and only if $E^{e+1}f = 0$ by \cite[\S 13.11]{00_Shimura}.
    A typical example is the weight two elliptic Eisenstein series, whose Fourier expansion is given by
    \[
    E_2(z) = \frac{3}{\pi y}- 1 + 24 \sum_{n=1}^\infty\left(\sum_{0<d \mid n} d\right) e^{2\pi \sqrt{-1}\, nz}, \qquad z \in \frakH_1.
    \]

    Let $\infty_1,\ldots,\infty_d$ be the archimedean places of $F$.
    For a congruence subgroup $\Gamma$ of $G_n(F)$, we regard $\Gamma$ as a subgroup of $G_n(\bbR)^d$ by $\gamma \mapsto (\infty_1(\gamma),\ldots,\infty_d(\gamma))$.
    Let $\rho$ be a finite dimensional representation of $\GL_n(\bbC)^d$.
    For $g = (g_1,\ldots,g_d) \in G_n(\bbR)^d$ and $z = (z_1,\ldots,z_d) \in \frakH^d_n$, the slash operator $|_\rho g$ is the automorphism of $C^\infty(\frakH_n^d,V)$ defined by
    \[
    (f|_{\rho}g)(z_1,\ldots,z_d) = \rho((j(g_i,z_i))_{i=1,\ldots,d})^{-1} f(g_1(z_1),\ldots,g_d(z_d)), \qquad 
    j\left(\abcd,z\right) = cz+d.
    \]
    By \cite[Proposition 12.10]{00_Shimura}, we have
    \[
    (D_{\rho}^ef)|_{\rho \otimes \tau^e} \gamma = D_{\rho \otimes \tau^e}(f|_{\rho} \gamma).
    \]
    We then say that a function $f \in C^\infty(\frakH_n^d, V)$ is a nearly holomorphic Hilbert-Siegel modular form of weight $\rho$ with respect to $\Gamma$ if $f$ satisfies the following three conditions:
    \begin{itemize}
        \item $f$ is nearly holomorphic;
        \item $f|_\rho\gamma = f$ for any $\gamma \in \Gamma$;
        \item $f$ satisfies the cusp condition.
    \end{itemize}
    Note that when $dn \geq 2$, the cusp condition is automatically satisfied by the Koecher principle.
    We denote by $N_\rho(\Gamma)$ the space of nearly holomorphic modular forms of weight $\rho$ with respect to $\Gamma$.
    The automorphy $f|_\rho \gamma = f$ yields the Fourier expansion of $f$
    \[
    f(z_1,\ldots,z_d) = \sum_{h \in \Sym_n(F)} c_h(y) \exp\left(2\pi \sqrt{-1}\, \sum_i\tr(\infty_i(h) \cdot z_i)\right), \qquad (z_1,\ldots,z_d) \in \frakH_n^d.
    \]
    Here, $c_h(y)$ is a polynomial in the entries of $y^{-1}$ with coefficients in $\bbC$.

    We consider nearly holomorphic automorphic forms.
    For each archimedean place $\infty_i$ of $F$, let $\frakp_{-, \infty_i}$ denote the Lie subalgebra of $\Lie(G_n(F_{\infty_i})) \otimes \bbC$.
    We say that an automorphic form $\varphi$ on $G_n(\bbA_F)$ is nearly holomorphic if $\varphi$ is $\frakp_{-,\infty_i}$-finite for any $i$.
    Let $\calN(G_n)$ be the space of nearly holomorphic automorphic forms on $G_n(\bbA_F)$.
    For an open compact subgroup $K_\fini$ of $G_n(\bbA_{F,\fini})$ and an irreducible finite dimensional representation $\rho = \boxtimes_{v \mid \infty}\rho_v$ of $\prod_{v \mid \infty} K_v$, let $\calN(G_n)_\rho^{K_\fini}$ be the space of $K_\fini$-fixed nearly holomorphic automorphic forms that generate $\rho$ as a representation of $\prod_{v \mid \infty} K_v$ under right translation $r$.
    When we denote the holomorphic extension of $\rho$ to $\GL_n(\bbC)^d$ by the same letter $\rho$, we obtain the canonical isomorphism
    \begin{align}\label{isom_NH_MF_AF}
    N_\rho(\Gamma) \otimes \rho^* \isom \calN(G_n)_\rho^{K_\Gamma},
    \end{align}
    where $K_\Gamma$ is the closure of $\Gamma$ in $G_n(\bbA_{F,\fini})$ under the diagonal embedding.
    
\subsection{Rationality of nearly holomorphic modular forms}

    Recall the Galois action on nearly holomorphic modular forms.
    Let $(\rho,V) = (\boxtimes_v\rho_v,\otimes_v V_v)$ be an irreducible representation of $\prod_{v\mid \infty} K_v$ with a rational structure $V_\bbQ = \otimes_v V_{v,\bbQ}$.
    Nearly holomorphic modular forms $f$ of weight $\rho$ can be expressed as
    \[
    f(z) = \sum_{h} p_h(\pi^{-1}\sqrt{-1}\ y_1^{-1},\ldots,\pi^{-1}\sqrt{-1}\ y_d^{-1}) \exp\left(2\pi\sqrt{-1}\, \sum_i \tr(\infty_i(h) \cdot z_i)\right).
    \]
    Here, $p_h$ lies in $\bigcup_e S_e(T,V) = \bigcup_e S_e(T,\otimes_{v\mid \infty} V_v)$.
    The subscript $i$ of $y_i^{-1}$ corresponds to the archimedean place $\infty_i$ of $F$.
    For an element $\sigma \in \Aut(\bbC)$, define $\sigma(i)$ by $\infty_{\sigma(i)} = \sigma \circ \infty_i$.
    By applying $\sigma$ to $S_e(T,V) = S_e(T_\bbQ, V_\bbQ) \otimes \bbC$, one obtains a polynomial ${^\sigma p_h}$.
    Set
    \[
    {^\sigma f}(z) = \sum_{h} {^\sigma p_h}(\pi^{-1}\sqrt{-1}\ y_{\sigma(1)}^{-1},\ldots,\pi^{-1}\sqrt{-1}\ y_{\sigma(d)}^{-1}) 
    \exp\left(2\pi\sqrt{-1}\, \sum_i \tr(\infty_i(h) \cdot z_i)\right)
    \]
    and
    \[
    {^\sigma \rho}((g_v)_v)(w_v)_v = (\rho_{v}(g_{\sigma \circ v})w_{v})_v
    \]
    for $(g_v)_v \in \prod_{v \mid \infty} K_v$ and $w_v \in V_v$.
    Note that the representation spaces of $\rho$ and ${^\sigma \rho}$ are the same.
    By \cite[Theorem 14.12 (1)]{00_Shimura}, the formal series ${^\sigma f}$ is a nearly holomorphic modular form in $N_{{^\sigma \rho}}$.
%    This coincides with the Galois action on the nearly holomorphic modular forms in \cite[\S 14.11]{00_Shimura}.
    For the representations $\rho \otimes \tau^e$ and $\rho \otimes \sigma^e$ on $S_e(T,V)$ of $\prod_{v \mid \infty}\GL_n(\bbC)$, we fix a rational structure by $S_e(\Sym_n(\bbQ), V_\bbQ)$.
    The Galois action commutes with the differential operators
    \begin{align}\label{relation_Gal_diff_op}
     {^\sigma\left(\frac{1}{\pi\sqrt{-1}}D_{v,\rho_v} f\right)} = \frac{1}{\pi\sqrt{-1}}D_{\sigma \circ v, {^\sigma\rho_v}} ({^\sigma f}), \qquad  {^\sigma\left(\pi \sqrt{-1}\ E_v f\right)} = \pi\sqrt{-1} \cdot E_{\sigma \circ v} ({^\sigma f})
    \end{align}
    as in the proof of \cite[Theorem 14.12 (3)]{00_Shimura}.

    We have defined the map $N_\rho \rightarrow N_{{^\sigma \rho}}$ associated with $\sigma \in \Aut(\bbC)$.
    For a rational structure $V_\bbQ$ of $\rho$, the dual $\rho^*$ is equipped with the rational structure $\Hom_\bbQ(V_\bbQ, \bbQ)$.
    Let $\sigma$ denote the isomorphism between $N_\rho \otimes \rho^*$ and $N_{{^\sigma \rho}} \otimes {^\sigma\rho}^{*}$ with respect to the rational structure $V_\bbQ \otimes \Hom_\bbQ(V_\bbQ, \bbQ)$ of $\rho \otimes \rho^*$.
    By the isomorphism \eqref{isom_NH_MF_AF}, one obtains an isomorphism $\sigma \colon \calN(G_n)_\rho \isom \calN(G_n)_{{^\sigma \rho}}$ such that the following diagram commutes:
    \begin{align}\label{def_Gal_act_AF}
  \begin{CD}
     N_\rho \otimes \rho^* @>{\sigma}>> N_{{^\sigma \rho}} \otimes {^\sigma\rho^{*}} \\
  @V{\eqref{isom_NH_MF_AF}}VV    @V{\eqref{isom_NH_MF_AF}}VV \\
     \calN(G_n)_\rho   @>{\sigma}>>  \calN(G_n)_{{^\sigma \rho}}
  \end{CD}
    \end{align} 
    
    \begin{lem}
        Let $(\rho, V)$ be an irreducible representation of $\prod_{v \mid \infty} K_v$.
        The Galois action defined in \eqref{def_Gal_act_AF} is independent of the choice of a rational structure of $V$.
    \end{lem}
    \begin{proof}
        Let $V_\bbQ$ and $V_\bbQ'$ be rational structures of $V$.
        As representations of $\prod_{v \mid \infty}\GL_n(\bbQ)$, $V_\bbQ$ decomposes as a direct sum of irreducible representations.
        By the isomorphisms $V_\bbQ \otimes \bbC \isom V$ and $V_\bbQ' \otimes \bbC \isom V$, the rational structures $V_\bbQ$ and $V_\bbQ'$ are irreducible and define an element $f \in \End_{\prod_{v \mid \infty}\GL_n(\bbQ)}(V) = \End_{\prod_{v \mid \infty}\GL_n(\bbC)}(V) = \bbC \cdot \mathrm{Id}_V$.
        In particular, the rational structure of $V$ is unique up to a nonzero constant.
        Hence, the map $\sigma \colon N_\rho \otimes \rho^* \rightarrow N_{{^\sigma \rho}} \otimes {^\sigma \rho^*}$ is independent of the choice of rational structures.
        By definition, $\calN(G_n)_{\rho} \rightarrow \calN(G_n)_{{^\sigma \rho}}$ is also independent of the choice of the rational structures.
    \end{proof}

    We next consider the relationship between the rational structure on $\calN(G_n)$ and the action of the Lie algebras $\frakp_{\pm,v}$.

    \begin{lem}\label{lem_arith_AF}
        For root vectors $X_{v,i,j}^+ \in \frakp_{+,v}$ and $X_{v,i,j}^- \in \frakp_{-,v}$, we have
        \[
        {^\sigma\left((\pi\sqrt{-1})^{\mp 1} X_{v,i,j}^\pm \cdot \varphi \right)} = (\pi\sqrt{-1})^{\mp1}X_{\sigma \circ v,i,j}^\pm \cdot {^\sigma \varphi}
        \]
        for $\varphi \in \calN(G_n)$.
    \end{lem}
    \begin{proof}
        It suffices to consider the case where $\varphi$ generates an irreducible representation $\rho = \boxtimes_{v \mid \infty} \rho_v$ of $\prod_{v \mid  \infty} K_v$.
        Then, the representations $\boxtimes_{w \neq v}\rho_w \boxtimes (\rho_v \otimes \tau^1)$ and $\boxtimes_{w \neq v}\rho_w \boxtimes (\rho_v \otimes \sigma^1)$ are multiplicity free.
        In the isomorphism \eqref{isom_NH_MF_AF}, we denote by $\varphi(f \otimes x^*)$ the automorphic form corresponding to $f \otimes x^* \in N_\rho \otimes \rho^*$.
        We may write $\varphi$ as $\varphi(f,x^*)$.
        Without loss of generality, we may further assume that $x^*$ is a pure tensor $x^* = \otimes_w x_w^*$ and lies in $V_\bbQ^*$.
        Using the chosen rational structure of $S_1(T,V)$ and the argument of \cite[Proposition 2.5]{HPSS}, we obtain
        \[
        X_{v,i,j}^+ \varphi = \varphi(D_{v,\rho_v} f, e_{v,i,j}^* \otimes x^*), 
        \qquad X_{v,i,j}^- \varphi = \varphi(E_v f, e_{v,i,j}^* \otimes x^*).
        \]
        Here, we regard $e_{v,i,j}^* \otimes x^*$ as an element of $\otimes_{v' \neq v} V_{v'}^* \otimes S_1(T_v,V_{v}^*) = S_1(T_v,\bbC) \otimes V^*$.
        By the definition of the Galois action and \eqref{relation_Gal_diff_op}, we obtain
        \begin{align*}
        {^\sigma\left(\frac{1}{\pi\sqrt{-1}}X_{v,i,j}^+ \varphi\right)}
        &= \varphi({^\sigma(}(\pi\sqrt{-1})^{-1}D_{v,\rho_v} f), e_{\sigma \circ v,i,j}^* \otimes \otimes_{w\mid \infty} x_{\sigma \circ w}^*))\\
        &=\varphi((\pi\sqrt{-1})^{-1}D_{\sigma \circ v, \rho_{\sigma \circ v}} ({^\sigma f}), e_{\sigma \circ v,i,j}^* \otimes \otimes_{w\mid \infty} x_{\sigma \circ w}^*)\\
        &=\frac{1}{\pi\sqrt{-1}}X_{\sigma \circ v,i,j}^+ \cdot {^\sigma \varphi},
        \\ {^\sigma\left(\pi\sqrt{-1} \cdot X_{v,i,j}^- \varphi\right)}
        &= \varphi(^\sigma(\pi\sqrt{-1} \cdot E_v f), e_{\sigma \circ v,i,j}^* \otimes \otimes_{w\mid \infty} x_{\sigma \circ w}^*)\\
        &=
        \varphi(\pi\sqrt{-1}\cdot E_{\sigma \circ v} ({^\sigma f}), e_{\sigma \circ v,i,j}^* \otimes \otimes_{w \mid \infty} x_{\sigma \circ w}^*) \\
        &= \pi\sqrt{-1} \cdot X_{\sigma \circ v,i,j}^- \cdot {^\sigma \varphi}.
        \end{align*}
        The last equalities follow from an analogue of \cite[Proposition 2.5]{HPSS} (see also \cite[(A.8.6a) and (A.8.6b)]{00_Shimura}).
    \end{proof}

\subsection{Galois action for automorphic representations}\label{subsec_Gal_act_on_AR}
Let $v$ be a nonarchimedean place of $F$ and let $\pi_v$ be an irreducible representation on $V_{\pi_v}$.
Following \cite[\S 3.1]{90_Clozel}, for $\sigma \in \Aut(\bbC)$, we define ${^\sigma\pi_v}$ to be the representation on $V_{\pi_v} \otimes_{\bbC,\sigma}\bbC$.
For a Hilbert-Siegel cusp form $f$ of degree $n$, let $\pi_f$ be the cuspidal automorphic representation of $G_n(\bbA_F)$ generated by $f$.
For $\sigma \in \Aut(\bbC)$, define $\tau_\sigma = (\tau_{\sigma,p})_p \in \widehat{\bbZ}^\times \subset \widehat{\calO}_{F}^\times$ by the element corresponding to $\sigma$ via the restriction map $\Aut(\bbC) \rightarrow \Gal(\bbQ_\ab/\bbQ) \isom \widehat{\bbZ}^\times$, i.e., $\sigma(\psi_F(x)) = \psi_F(\tau_\sigma x)$ for any $x \in \bbA_{F, \fini}$.

We describe the action on ${^\sigma f}$.
Let $\varphi$ be an automorphic form corresponding to $f$ and let $K$ be an open compact subgroup of $G_n(\bbA_{F,\fini})$ that fixes $\varphi$.
For $\gamma \in G_n(F)$, by strong approximation, there exists ${^\sigma \gamma} \in G_n(F)$ such that ${^\sigma \gamma} \in \mathop{\mathrm{Ad}}(\diag(\tau_{\sigma,v}\cdot 1_n, 1_n))(\gamma) K G_n(\bbR)$.

\begin{lem}\label{lem_Gal_action_MF}
    For a Hilbert-Siegel modular form $f$ of weight $\rho$ and $\sigma \in \Aut(\bbC)$, we have
    \[
    {^\sigma(f|_{\rho} \gamma)} = {^\sigma f}|_{{^\sigma \rho}} {^\sigma\gamma}.
    \]
\end{lem}
\begin{proof}
    See \cite[Lemma 10.5]{00_Shimura} and its proof.
\end{proof}
As a direct consequence, we may explicitly determine the automorphic representation generated by a cusp form ${^\sigma f}$.
\begin{cor}
    For a nonarchimedean place $v$, a Hilbert-Siegel cusp form $f$ and $\sigma \in \Aut(\bbC)$, the $v$-component of the cuspidal automorphic representation $\pi_{{^\sigma f}}$ is given by $g \mapsto {^\sigma \pi_v}(\Ad(\tau_{\sigma,v} \cdot 1_n, 1_n)(g))$.
\end{cor}
\begin{rem}\label{rem_rep_gen_by_f^sigma}
    Since $\diag(\tau_{\sigma,v} \cdot 1_n, 1_n) \in \GSp_{2n}(F_v)$, when we consider the automorphic representations generated by $\GSp_{2n}(\bbA_F)$, the $v$-component of the cuspidal automorphic representation of $\GSp_{2n}(\bbA_F)$ generated by ${^\sigma f}$ is isomorphic to the representation ${^\sigma\pi_v}$ of $\GSp_{2n}(F_v)$.
\end{rem}

We will later consider the principal congruence subgroup $\Gamma(\frakN)$ of level $\frakN$, where $\frakN$ is an integral ideal of $\calO_F$.
This is defined by
\[
\Gamma(\frakN) = \{g \in G_n(\calO_F) \mid g \equiv 1_{2n} \bmod \frakN\}.
\]
For a nonarchimedean place $v$, the closure $K_v(\frakN)$ of the image of $\Gamma(\frakN)$ in $G_n(F_v)$ is a normal subgroup of $G_n(\calO_v)$.
The open compact subgroup $K_v(\frakN)$ is stable under the adjoint action of $\diag(\tau_{\sigma,v} \cdot 1_n, 1_n)$.
Hence, by Lemma \ref{lem_Gal_action_MF}, the map defined by $\sigma \in \Aut(\bbC)$ on $
\calN(G_n)_\rho$ as in \eqref{def_Gal_act_AF} preserves the $\prod_{v < \infty} K_v(\frakN)$-fixed vectors.
\section{Local zeta integrals and \texorpdfstring{$L$}{}-functions}\label{section_zeta_int}
% We review Tate's local factors and the standard local factors of $G_n(F_v)$ via the doubling integral \cite{GPSR_87}.
% Using the functional equation established in \cite{2005_Lapid-Rallis}, we calculate the local zeta integrals, which play a crucial role in this paper.

\subsection{Tate's local \texorpdfstring{$L$}{}-factors}
    We first recall Tate's local factors \cite{1979_tate}.
    Let $v$ be a place of $F$.
    For a unitary character $\mu_v$ of $F_v^\times$, let $L_v(s,\mu_v)$ be the local $L$-factor for $\mu_v$ defined in \cite[\S 3]{1979_tate}.
    For the nontrivial additive character $\psi_v$, the $\vep$-factor is the unique invertible function such that
    \[
    \frac{\int_{F_v^\times}{\widehat{\varphi}(x)} \mu_v^{-1}(x)|x|_v^{-s}\ d_{\psi_v} x}{L_v(1-s,\mu_v^{-1})}
    =
    \vep_v(s, \mu_v, \psi_v)
    \frac{\int_{F_v^\times}{\varphi(x)} \mu_v(x)|x|_v^{s-1}\ d_{\psi_v} x}{L_v(s,\mu_v)}
    \]
    for every Schwartz function $\varphi$ on $F_v$.
    Here, $d_{\psi_v} x$ is the selfdual measure with respect to $\psi_v$ and $\widehat{\varphi}$ is the Fourier transform of $\varphi$ with respect to $d_{\psi_v} x$.
    The function
    \begin{equation}\label{FE_Tate}
    \gamma_v(s,\mu_v,\psi_v) = \vep_v(s,\mu_v,\psi_v) \frac{L_v(1-s, \mu_v^{-1})}{L_v(s,\mu_v)}
    \end{equation}
    is called the local $\gamma$-factor of $\mu_v$ with respect to $\psi_v$.
    It satisfies the functional equation
    \[
    \gamma_v(s,\mu_v,\psi_v) \gamma_v(1-s,\mu_v^{-1},\psi_v) = \vep_v(s,\mu_v,\psi_v) \vep_v(1-s,\mu_v^{-1},\psi_v) =\mu_v(-1).
    \]
    If we replace $\psi_v$ by $\psi_{v,a}$, then
    \[
    \vep_v(s, \mu_v, \psi_{v,a}) = \mu_v(a)|a|_v^{s-1/2} \vep_v(s,\mu_v,\psi_v).
    \]
    For $\psi_{v,\bdelta_v^{-1}}$ of order $0$ and $c \in F_v^\times$ such that $c \calO_v$ is the conductor of $\mu_v$, we have
    \begin{equation}\label{def_epsilon_factor}
    \vep_v(s, \mu_v, \psi_{v,\bdelta_v^{-1}})
    =
    \begin{dcases}
        1 & \text{if $\mu_v$ is unramified;}\\
        \int_{c^{-1}\calO_v^\times} \mu_v^{-1}(x)|x|_v^{-s}\psi_{v,\bdelta_v^{-1}}(x)\ dx & \text{if $\mu_v$ is ramified.}
    \end{dcases}
    \end{equation}
    From this, we obtain $\vep_v(s,\mu_v,\psi_{v,\bdelta_v^{-1}}) = q_v^{-c(\mu_v) s} \vep_v(0,\mu_v, \psi_{v,\bdelta_v^{-1}})$, where $c(\mu_v)$ is the conductor exponent of $\mu_v$.
    For later use, we consider the relationship between the Weil constant and the epsilon factor.
    Let $\alpha_{\psi_v}(x)$ be the Weil constant associated with the additive character $\psi_v$ and $x \in F_v^\times$ as in \cite{2017_Ikeda}.
    Note that $\alpha_{\psi_{v,\bdelta_v^{-1}}}(x) = \alpha_{\psi_v}(\bdelta_v^{-1}x)$.

    \begin{lem}\label{lem_weil_const_epsilon_factor}
        For $\theta \in F^\times_v/F_v^{\times 2}$,
        we have
        \[
        \frac{\alpha_{\psi_{v,\bdelta_v^{-1}}}(1)}{\alpha_{\psi_{v,\bdelta_v^{-1}}}(\theta)} = \vep_v(1/2, \chi_\theta,\psi_{v,\bdelta_v^{-1}}).
        \]
    \end{lem}
    \begin{proof}
        By \cite[Lemma 1.3]{2017_Ikeda}, for a character $\omega$ of $F_v^\times$, we have
        \[
        \sum_{\theta' \in F_v^\times/F_v^{\times2}} \frac{\alpha_{\psi_{v,\bdelta_v^{-1}}}(1)}{\alpha_{\psi_{v,\bdelta_v^{-1}}}(\theta')}
        \gamma_v(s,\omega\chi_{\theta'}, \psi_{v,\bdelta_v^{-1}})^{-1}
        =
        2|2|_v^{-2s}\omega^{-1}(4) \gamma_v(2s,\omega^2,\psi_{v,\bdelta_v^{-1}})^{-1}
        \gamma_v(s+1/2, \omega, \psi_{v,\bdelta_v^{-1}}).
        \]
        Substituting $\omega = \chi_\theta$ with $\omega \neq \mathbf{1}$ and taking the residue at $s=0$, we have
        \[
        \frac{\alpha_{\psi_{v,\bdelta_v^{-1}}}(1)}{\alpha_{\psi_{v,\bdelta_v^{-1}}}(\theta)} \cdot \vep_v(0,\mathbf{1},\psi_{v,\bdelta_v^{-1}})^{-1}\frac{1-q_v^{-1}}{\log(q_v)}
        =
        2 \cdot \vep_v(0,\mathbf{1},\psi_{v,\bdelta_v^{-1}})^{-1} \cdot \frac{1-q_v^{-1}}{2\log(q_v)} \cdot \vep_v(1/2,\chi_\theta,\psi_{v,\bdelta_v^{-1}})
        \]
        and hence
        \[
        \frac{\alpha_{\psi_{v,\bdelta_v^{-1}}}(1)}{\alpha_{\psi_{v,\bdelta_v^{-1}}}(\theta)} 
        = 
        \vep_v(1/2, \chi_\theta,\psi_{v,\bdelta_v^{-1}}).
        \]
        The case $\theta \in F^{\times 2}_v$ is immediate by $\vep_v(1/2, \chi_\theta,\psi_{v,\bdelta_v^{-1}}) = \vep_v(1/2, \mathbf{1},\psi_{v,\bdelta_v^{-1}}) = 1$.
    \end{proof}

\subsection{Degenerate principal series representations}
    Let $v$ be a place of $F$.
    The degenerate principal series representation is defined to be the normalized induced representation
    \[
    I_{v,n}(s,\mu_v) = \Ind_{P_n(F_v)}^{G_{n}(F_v)} ((\mu_v\cdot | \ |_{v}^s)\circ\Delta), \qquad I_{n}(s,\mu) = \otimes_v I_{v,n}(s,\mu_v).
    \]

    For a function $(s,g) \mapsto f_{v,s}(g)$ on $\bbC \times G_n(F_v)$ (or $f_{v,s}$ for short), we say that $f_{v,s}$ is standard if the restriction $f_{v,s}|_{K_v}$ is independent of $s$.
    We say that a standard section $f_{v,s}$ is $K_v$-finite if $f_{v,s_0}$ is $K_v$-finite for some $s=s_0$.
    By a holomorphic section (resp.\ meromorphic section), we mean a finite sum of standard sections with coefficients in holomorphic functions (resp.~meromorphic functions).

    Put
    \[
    w_n = 
    \begin{pmatrix}
        0_n & 1_n\\
        -1_n & 0_n
    \end{pmatrix}.
    \]
    We define the intertwining operator $M_{v,w_n,s} \colon I_{v,n}(s,\mu_v) \rightarrow I_{v,n}(-s,\mu_v^{-1})$ by
    \begin{align}\label{def_int_op}
    M_{v,w_{n},s}f_s(g) = \int_{N_n(F_v)} f_s(w_n ng)\, dn.
    \end{align}
    Here, $dn$ is the Haar measure of $N_n(F_v)$ so that $\vol(N_n(\calO_v))=1$.
    The defining integral in \eqref{def_int_op} converges absolutely for $\mathrm{Re}(s) > (n+1)/2$ and can be continued meromorphically to all $s \in \bbC$.
    Set
    \[
    a_{v,n}(s,\mu_v) = L_v(s-(n-1)/2, \mu_v)\prod_{j=1}^{\lfloor n/2\rfloor} L_v(2s-n+2j, \mu_v^2)
    \]
    and
    \[
    b_{v,n}(s, \mu_v) = L_v(s+(n+1)/2, \mu_v)\prod_{j=1}^{\lfloor n/2\rfloor} L_v(2s+n+1-2j, \mu^2_v).
    \]
    Put $a_{n}^S(s,\mu) = \prod_{v \not \in S} a_{v,n}(s,\mu_v),b_{n}^S(s,\mu) = \prod_{v \not \in S} b_{v,n}(s,\mu_v)$ and $M_{v,w_n,s}^* = a_{v,n}(s,\mu_v)^{-1}M_{v,w_n,s}$.
    \begin{prop}[{\cite[Lemma 3.1]{2014_Yamana_theta}}]\label{prop_holomorphy_M^*}
        The operator $M_{v,w_n,s}^*$ is nonzero holomorphic for any $s$.
    \end{prop}
    
    We consider the intertwining operator $M_{v,w_{2n},s}$ for $G_{2n}$ and the functional equation.
    For $B \in \Sym_{2n}(F_v)$ with $\det B \neq 0$, we define a generalized Whittaker functional
    \[
    \ell_{\psi_{v}, B}(f_{v,s}) = \int_{\Sym_{2n}(F_v)} f_{v,s}
   (w_{2n}\bfn(x)) \psi_{v}(\tr(Bx))\ dx
    \]
    for a section $f_{v,s}$ of $I_{v,2n}(s,\mu_v)$.
    The degenerate Whittaker integral $\ell_{\psi_{v}, B}$ converges absolutely for $\mathrm{Re}(s)> n + 1/2$ and can be continued holomorphically to all $s$ by \cite{1979_Karel, 1988_Wallach}.
    Then, there exists a meromorphic function $c(s,\mu_v,B,\psi_{v})$ such that
    \begin{align}\label{FE_Siegel_series}
    \ell_{\psi_{v}, B} 
    \circ M_{v,w_{2n},s} = \mu_v(\det(B))^{-1}|\det(B)|^{-s}_v c(s,\mu_v,B,\psi_{v}) \cdot \ell_{\psi_{v}, B}.
    \end{align}
    Note that for $\lambda \in F_v^\times$, the function $c(s,\mu_v,B,\psi_{v, \lambda})$ equals
    \[
    c(s,\mu_v,B,\psi_{v, \lambda}) 
    =
    \mu_v(\lambda)^{-2n}|\lambda|_v^{-2ns}
    c(s,\mu_v,B,\psi_{v})
    \]
    and $c(s,\mu_v,B,\psi_{v, \lambda}) = c(s,\mu_v,B,\psi_{v, \lambda}^{-1})$.

    \begin{thm}[{\cite[Theorem 2.1]{2017_Ikeda}}]\label{Whittaker_Int_op}
        The function $c(s,\mu_v, B,\psi_{v, \bdelta_v^{-1}})$ is given by
        \begin{align*}
        c(s,\mu_v, B,\psi_{v, \bdelta_v^{-1}})
        &=
        \gamma_v(s-n+1/2, \mu_v, \psi_{v, \bdelta_v^{-1}})^{-1}
        \prod_{j=1}^{n} \gamma_v(2s-2n+2j, \mu_v^2,\psi_{v, \bdelta_v^{-1}})^{-1}\\
        &\qquad \cdot |2|_v^{-2ns}\mu_v(4)^{-n}
        \vep_v(1/2, \chi_{D_B}, \psi_{v,\bdelta_v^{-1}})^{-1}
        \gamma_v(s+1/2, \mu_v\chi_{D_B},\psi_{v, \bdelta_v^{-1}}).
        \end{align*}
    \end{thm}
    \begin{proof}
        Note that $\ell_{\psi_v,B}$ differs from the one used in \cite[\S 3]{2017_Ikeda}; the associated additive character is the complex conjugate of the one used there. We also note that the measure on $N_{2n}(F_v)$ is different from the one used in \cite{2017_Ikeda}.
        The statement follows from 
        \cite[Theorem 2.1]{2017_Ikeda} and Lemma \ref{lem_weil_const_epsilon_factor}.
    \end{proof}
    
    For $B \in \Sym_{2n}(F_v)$ with $\det(B) \neq 0$, we define the normalization factor
    \[
    C(s,\mu_v, \psi_{v,\bdelta_v^{-1}})
    =
    c(s,\mu_v, B, \psi_{v,\bdelta_v^{-1}})
    \frac{\mu_v(2)^{2n}|2|_v^{2ns}\vep_v(1/2, \chi_{D_B}, \psi_{v,\bdelta_v^{-1}})}{\gamma_v(s+1/2,\mu_v\chi_{D_B}, \psi_{v,\bdelta_v^{-1}})}
    \]
    and the normalized intertwining operator
    \[
    M^\dagger_{v,w_{2n},s} = C(s,\mu_v,\psi_{v,\bdelta_v^{-1}})^{-1}M_{v,w_{2n},s}.
    \]
    The normalization factor $C(s,\mu_v, \psi_{v,\bdelta_v^{-1}})$ is given by
    \[
    C(s,\mu_v, \psi_{v,\bdelta_v^{-1}}) =
    \gamma_v(s-n+1/2, \mu_v, \psi_{v, \bdelta_v^{-1}})^{-1}
    \prod_{j=1}^{n} \gamma_v(2s-2n+2j, \mu_v^2,\psi_{v, \bdelta_v^{-1}})^{-1} 
    \]
    and independent of $B$.
    By \eqref{FE_Tate} and \eqref{FE_Siegel_series}, $M_{v,w_{2n},s}^\dagger$ satisfies the functional equation
    \[
    M^\dagger_{v,w_{2n},s}\circ
    M^\dagger_{v,w_{2n},-s} = \mu_v(-1) \cdot \mathrm{Id}.
    \]
    
\subsection{Local zeta integrals and Lapid-Rallis' local factors}\label{LR_local_factors}
    Fix a positive integer $n$.
    Recall the doubling embedding of $G_n \times G_n$ into $G_{2n}$.
    For $\tau_v \in \calO_{v}^\times$, set 
    \[
    \delta=
    \left(
    \begin{array}{cc|cc}
        1_n & 0_n & 0_n & 0_n \\
        0_n & 1_n & 0_n & 0_n \\
        \hline
        0_n & 1_n & 1_n & 0_n\\
        1_n & 0_n & 0_n & 1_n
    \end{array}
    \right), \qquad 
    \delta_{\tau_{v}} =
    \left(
    \begin{array}{cc|cc}
        1_n & 0_n & 0_n & 0_n \\
        0_n & 1_n & 0_n & 0_n \\
        \hline
        0_n & \tau_{v} \cdot 1_n & 1_n & 0_n\\
        \tau_{v} \cdot 1_n & 0_n & 0_n & 1_n
    \end{array}
    \right).
    \]
    We define an embedding $\iota \colon G_n \times G_n \rightarrow G_{2n}$ by
    \[
    \iota 
    \left(
    \begin{pmatrix}
        a_1&b_1\\
        c_1&d_1
    \end{pmatrix},
    \begin{pmatrix}
        a_2 & b_2 \\
        c_2 & d_2
    \end{pmatrix}
    \right) 
    = 
    \left(
    \begin{array}{cc|cc}
        a_1 & 0 & b_1 & 0 \\
        0 & a_2 & 0 & b_2 \\
        \hline
        c_1 & 0 & d_1 & 0 \\
        0 & c_2 & 0 & d_2
    \end{array}
    \right).
    \]
    Note that $\delta \cdot \iota(g,\check{g})\cdot \delta^{-1}$ lies in $P_{2n}$ for any $g \in G_n$, where 
    \[
    \text{
    $\check{g} = 
    \begin{pmatrix}
        d & c \\
        b & a
    \end{pmatrix}\qquad$ for  $g = \abcd.$
    }
    \]

    Let $\pi_v$ be an irreducible unitary representation of $G_n(F_v)$ and $( \ , \ )$ be the canonical pairing on $\pi_v \otimes \pi_v^\vee$. For a holomorphic section $f_{v,s} \in I_{v,2n}(s, \mu_v)$, and for $w \in \pi_v, w^\vee \in \pi_v^\vee$, put
   \[
   T_v(s, f_{v,s})(w) = \int_{G_n(F_v)} f_{v,s} (\delta^{-1} \cdot \iota(g,1_{2n}))  \pi_v(g)w \, dg, \qquad 
    Z_v(s,f_{v,s},w,w^\vee) = (T_v(s,f_{v,s})(w), w^\vee).
    \]
    These integrals converge absolutely for ${\rm Re}(s) \gg 0$, and $Z_v(s,f_{v,s},w,w^\vee)$ admits a meromorphic continuation to the whole complex plane. 
    In the choice of $f_{v,s}$ in \S \ref{subsec_choice_global_section} below, $T_v(s,f_{v,s})(w)$ also admits a meromorphic continuation.
    % One of the purposes of this paper is to compute the zeta integral associated with a new meromorphic section explicitly.
    % The explicit computation below yields the arithmeticity of the standard $L$-values in the left half plane.

    \begin{rem}
        Let 
        \[
        h_\bullet = 
        \left(\begin{array}{cc|cc}
            1_n & 0_n & 0_n & -\frac{1}{2}1_n \\ 
            0_n & 1_n & -\frac{1}{2}1_n & 0_n \\ 
            \hline
            0_n & 1_n & \frac{1}{2}1_n & 0_n \\ 
            1_n & 0_n & 0_n & \frac{1}{2}1_n
        \end{array}\right).
        \]
        Then, we have
        \[
        Z(s, r(h_\bullet)f_s, w, w^\vee) = Z'(f_s, w, w^\vee)
        \]
        where $r(h_\bullet)f_s$ denotes the right translation of $f_s$ by $h_\bullet$, and $Z'(f_s, w, w^\vee)$ denotes the zeta integral of Piatetski-Shapiro and Rallis \cite{1986_PS-R}. 
    \end{rem}

\begin{defn}
We define the local $\gamma$-factors at $v$ to be the unique family of meromorphic functions $\gamma_v(s, \pi_v, \mu_v, \psi_v)$ in $s$ parameterized by an irreducible representation $\pi_v$ of $G_n(F_v)$, a character $\mu_v$ of $F_v^\times$, and a nontrivial additive character $\psi_v$ of $F_v$.
The local $\gamma$-factors are characterized by the ten properties \cite[Theorem 4]{2005_Lapid-Rallis}. Moreover, we define the local $L$-factors $L_v(s,\pi_v,\mu_v)$ and the local $\varepsilon$-factors $\varepsilon_v(s, \pi_v, \mu_v, \psi_v)$ as in \cite[\S10]{2005_Lapid-Rallis}.
\end{defn}

These three factors satisfy the relation
\[
\gamma_v(s+1/2, \pi_v^\vee, \mu_v, \psi_v) = \varepsilon_v(s+1/2, \pi_v^\vee, \mu_v, \psi_v)\frac{L_v(-s+1/2, \pi_v, \mu^{-1}_v)}{L_v(s+1/2, \pi_v^\vee, \mu_v)}.
\]
When we replace $\psi_{v}$ by $\psi_{v,\lambda}$, one has
\[
\gamma_v(s, \pi_v, \mu_v, \psi_{v,\lambda}) = (\mu_v(\lambda) |\lambda|_v^{s-1/2})^{2n+1} \gamma_v(s,\pi_v,\mu_v,\psi_v).
\]
It is known that the local factors are selfdual, that is
\begin{equation}\label{selfduality_local_factor}
\gamma_v(s,\pi_v,\mu_v,\psi_v) = \gamma_v(s,\pi_v^\vee,\mu_v,\psi_v),\ L_v(s,\pi_v,\mu_v) = L_v(s,\pi_v^\vee,\mu_v), \ \vep_v(s,\pi_v,\mu_v,\psi_v) = \vep_v(s,\pi_v^\vee,\mu_v,\psi_v).
\end{equation}
The analytic construction of the local $\gamma$-factor shows that the doubling zeta integral satisfies the following local functional equation.
    \begin{thm}[{\cite{2005_Lapid-Rallis}}]\label{th_functional_eq_zeta_int}
        Let $\pi_v$ be an irreducible representation of $G_n(F_v)$ with central character $\omega_{\pi_v}$ and let $\mu_v$ be a character of $F^\times_v$.
        Take $w \in \pi_v$ and $w^\vee \in \pi^\vee_v$.
        For a meromorphic section $f_{v,s}$ of $I_{v,2n}(s,\mu_v)$, the following functional equation holds:
        \begin{align*}
        \frac{Z_v(-s,M^\dagger_{v,w_{2n},s}f_{v,s}, w,w^\vee)}{L_v(1/2-s, \pi_v, \mu_v^{-1})} &= \omega_{\pi_v}(-1)\vep_v(s+1/2, \pi_v^\vee, \mu_v, \psi_{v,\bdelta_v^{-1}}) \frac{Z_v(s,f_{v,s}, w,w^\vee)}{L_v(s+1/2, \pi_v^\vee, \mu_v)}.
        \end{align*}
    \end{thm}

\subsection{Coincidence with the standard local factors}\label{LR_vs_std}
    
Since the local Langlands correspondence for symplectic groups has been established \cite{2013_Arthur}, it is natural to ask whether the local factors of Lapid-Rallis defined in \S\ref{LR_local_factors} coincide with the standard factors defined via the $L$-parameter. If $\pi_v$ and $\mu_v$ are unramified, then the Lapid-Rallis $L$ and $\varepsilon$-factors are defined to be the standard factors:
\[
L_v(s, \pi_v,\mu_v) = L_v(s, \pi_v, \mu_v, {\rm std}), \quad   \varepsilon_v(s, \pi_v, \mu_v, \psi_v) = \varepsilon_v(s, \pi_v, \mu_v, \psi_v, {\rm std}).
\] 
In this subsection, we verify this coincidence for arbitrary irreducible representations.

Let $K$ be a local field of characteristic $0$, let $\pi$ be an irreducible representation of $G_{n}(K)$, let $\mu$ be a character of $K^\times$, and let $\psi$ be a nontrivial additive character of $K$. By the Langlands classification, both the Lapid-Rallis factors and the standard factors for general $\pi$ are determined by those of essentially tempered representations via parabolic induction. Thus, we may assume that $\pi$ is tempered. In this case, both factors $L(s, \pi, \mu)$ (resp. $\varepsilon(s, \pi, \mu, \psi)$) and $L(s, \pi, \mu, {\rm std})$ (resp. $\varepsilon(s, \pi, \mu, \psi, {\rm std})$) can be reconstructed from the local $\gamma$-factors
\begin{align*}
\gamma(s, \pi, \mu, \psi) 
&= \frac{L(1-s, \pi^\vee,\mu^{-1})}{L(s, \pi,\mu)}\varepsilon(s, \pi, \mu, \psi), \\
\gamma(s, \pi, \mu, \psi, {\rm std}) 
&= \frac{L(1-s, \pi^\vee, \mu^{-1}, {\rm std})}{L(s, \pi, \mu, {\rm std})}\varepsilon(s, \pi, \mu, \psi, {\rm std}).
\end{align*}
One can show that both $\gamma$-factors $\gamma(s, \pi, \mu, \psi)$ and $\gamma(s, \pi, \mu, \psi, {\rm std})$ have the same multiplicativity under parabolic induction from discrete series representations. 
Hence, it suffices to show  
\begin{align}\label{gamma, LR = std} 
 \gamma(s, \pi, \mu, \psi) =  \gamma(s, \pi, \mu, \psi, {\rm std})  
 \end{align} 
when $\pi$ is an irreducible discrete series representation. 
For archimedean places, the equation \eqref{gamma, LR = std} was proved in \cite[\S9]{2005_Lapid-Rallis}.
For nonarchimedean places, this follows from the standard global argument using Arthur's globalization result \cite[Lemma 6.2.2]{2013_Arthur}.

\subsection{Computations at nonarchimedean places}
\label{subsec_local_zeta_int_nonarch}

    Let $\pi_v$ be an irreducible representation of $G_n(F_v)$ and $\mu_v$ be a character of $F_v^\times$.
    Let $f_{v,s}$ be a holomorphic section of $I_{v,2n}(s,\mu_v)$.
    We first consider the unramified case, i.e., $f_{v,s}, \pi_v$ and $\mu_v$ are unramified.
    Take unramified vectors $w \in \pi_v$ and $w^\vee \in \pi_v^\vee$. 
    Since the $\vep$-factor coincides with the Langlands $\vep$-factor, we have $\vep_v(s,\pi_v,\mu_v,\psi_{v,\bdelta^{-1}}) = 1$.
    We have the following well-known formula (cf. \cite[Theorem 1.2]{1986_PS-R}, \cite[Proposition 3]{2005_Lapid-Rallis} and \cite[Theorem 19.8]{00_Shimura}).

    \begin{thm}
    \label{thm_unr_zeta_int}
        Let $f_{v,s}$ be an unramified standard section of $I_{v,2n}(s,\mu_v)$ with $f_{v,s}(1_{4n}) = 1$.
 Then
        \[
        Z_v(s,f_{v,s}, w,w^\vee)
        =
        \frac{L_v(s+1/2, \pi_v^\vee,\mu_v)}{b_{v,2n}(s, \mu_v)} (w,w^\vee).
        \]
    \end{thm}

    In the general case, take $m_v \in \bbZ_{>0}$ so that $\pi_v^{K_v(\frakp_v^{m_v})} \neq 0$ and $\mu_v|_{1+\frakp_v^{m_v}} = \mathbf{1}$.
    We also allow the case in which $\pi_v$ and $\mu_v$ are unramified.
    We define the unique section $\bff_{v,s}$ of $I_{v,2n}(s,\mu_v)$ by
    \[
    \bff_{v,s}(g) = 
    \begin{cases}
        0 & \text{if $g \not \in P_{2n}(F_v)K_v(\frakp^{m_v})$;}\\
        \mu_v(\Delta(p))|\Delta(p)|_v^{s+n+1/2} & \text{if $g = pk \in P_{2n}(F_v)K_v(\frakp_v^{m_v})$.}
    \end{cases}
    \]
    When we emphasize that $\bff_{v,s}$ is a section of $I_{v,2n}(s,\mu_v)$, we write it as $\bff_{v}(s,\mu_v)$. 
    The section $\bff_{v,s}$ is essentially the same as the one constructed in \cite[(16.26)]{00_Shimura}.

    \begin{lem}\label{lem_zeta_int_ramified_big_cell}
        For $\tau_v \in \calO_{v}^\times, w \in \pi_v^{K_v(\frakp_v^{m_v})}$ and $w^\vee \in \pi^{\vee, K_v(\frakp_v^{m_v})}$, we have
        \[
        T_v(s,r(\delta_{\tau_v})\bff_{v,s})(w) = \vol(K_v(\frakp_v^{m_v})) \cdot \mu_v(\tau_v)^{n}\pi_v(\diag(\tau_v \cdot 1_n, \tau_v^{-1}\cdot 1_n))w
        \]
        and
        \[
        Z_v(s,r(\delta_{\tau_v})\bff_{v,s}, w,w^\vee) = \vol(K_v(\frakp_v^{m_v})) \cdot\mu_v(\tau_v)^{n}(\pi_v(\diag(\tau_v \cdot 1_n, \tau_v^{-1}\cdot 1_n))w,w^\vee).
        \]
    \end{lem}
    \begin{proof}
        By definition, we have
        \begin{align*}
        &T_v(s,r(\delta_{\tau_v})\bff_{v,s})(w)\\
        &=
        \int_{G_n(F_v)}
        \bff_{v,s}(\delta^{-1} \cdot \iota(g,1_{2n})\delta_{\tau_v})
        \pi_v(g)w \, dg\\
        &=
        \int_{G_n(F_v)}
        \bff_{v,s}(\diag(\tau_v \cdot 1_n,1_n, \tau_v^{-1}\cdot 1_n,1_n)\delta^{-1}_{\tau_v} \cdot \iota(g,1_{2n})\delta_{\tau_v})
        \pi_v(\diag(\tau_v \cdot 1_n, \tau_v^{-1}\cdot 1_n)g)w \, dg\\
        &=
        \mu_v(\tau_v)^{n}\int_{G_n(F_v)}
        \bff_{v,s}(\delta^{-1}_{\tau_v} \cdot \iota(g,1_{2n})\delta_{\tau_v})
        \pi_v(\diag(\tau_v \cdot 1_n, \tau_v^{-1}\cdot 1_n)g)w \, dg.
        \end{align*}
        We determine the condition on $g$ under which
        \[
        \delta^{-1}_{\tau_v} \cdot \iota(g,1_{2n})\delta_{\tau_v} \in \Supp(\bff_{v,s}) = P_{2n} K_v(\frakp_v^{m_v}).
        \]
        By 
        \begin{align}\label{def_ad_delta_iota}
        \delta^{-1}_{\tau_v} \cdot \iota(g,1_{2n})\delta_{\tau_v} = 
        \left(
        \begin{array}{cc|cc}
            a & b\tau_v & b & 0_n\\
            0_n & 1_n & 0_n&0_n\\
            \hline
            c & (d-1_n)\tau_v & d & 0_n\\
            \tau_v(1_n-a) & -\tau_v^2 b & -\tau_v b & 1_n
        \end{array}\right), \qquad
        g = \abcd,
        \end{align}
        and the explicit description of the matrices in $P_{2n}(F_v)K_v(\frakp_v^{m_v})$, the lower right $2n \times 2n$-block of \eqref{def_ad_delta_iota} is invertible.
        Moreover, by the same argument as in \cite[Lemma 4.2]{21_PSS_algebraicity}, the matrix $g$ lies in $K_v(\frakp_v^{m_v})$.
        Hence, we obtain
        \begin{align*}
        T_v(s,r(\delta_{\tau_v})\bff_{v,s})(w)
        &=
        \mu_v(\tau_v)^{n}
        \pi_v(\diag(\tau_v \cdot 1_n, \tau_v^{-1}\cdot 1_n))\left(
        \int_{K_v(\frakp_v^{m_v})}
        \bff_{v,s}(\delta^{-1}_{\tau_v} \cdot \iota(g,1_{2n})\delta_{\tau_v})
        \pi_v(g)w\, dg\right)\\
        &=\vol(K_v(\frakp_v^{m_v})) \cdot\mu_v(\tau_v)^{n}\pi_v(\diag(\tau_v \cdot 1_n, \tau_v^{-1}\cdot 1_n))w.
        \end{align*}
        In the last line, for $g \in K_v(\frakp_v^{m_v})$, using the description \eqref{def_ad_delta_iota}, we obtain $\bff_{v,s}(\delta_{\tau_v}^{-1} \cdot \iota(g,1_{2n})\delta_{\tau_v}) = 1$.
        The assertion for $Z_v(s,r(\delta_{\tau_v})\bff_{v,s}, w,w^\vee)$ follows from the definition.
    \end{proof}

    Define a meromorphic section $\bfh_{v,s}$ of $I_{v,2n}(s,\mu_v)$ by
    \[
    \bfh_{v,s} = M_{v,w_{2n},-s} (\bff_{v}(-s,\mu_v^{-1})).
    \]
    % We summarize the key properties of $\bfh_{v,s}$ as follows:
    \begin{lem}\label{lem_ram_zeta_int}
        With the above notation, we have
        \begin{enumerate}
            \item $\bfh_{v,s}$ lies in the kernel of $M^*_{v,w_{2n},s}$ at $s=1/2, 3/2,\ldots,n-1/2$ if $\mu_v^2 = \mathbf{1}$.
            \item For $\tau_v \in \calO_{v}^\times, w \in \pi_v^{K_v(\frakp_v^{m_v})}$ and $w^\vee \in \pi_v^{\vee, K_v(\frakp_v^{m_v})}$, we have
            \begin{align*}
        &T_v(s,r(\delta_{\tau_v})\bfh_{v,s})(w) \\
        &= \vol(K_v(\frakp_v^{m_v})) \cdot\gamma_v(-s-n+1/2, \mu_v^{-1}, \psi_{v,\bdelta_v^{-1}})^{-1}
        \prod_{j=1}^{n}\gamma_v(-2s-2n+2j, \mu_v^{-2}, \psi_{v,\bdelta_v^{-1}})^{-1}\\
        &\qquad\cdot \omega_{\pi_v}(-1)\gamma_v(-s+1/2, \pi_v^\vee, \mu_v^{-1}, \psi_{v,\bdelta_v^{-1}}) \mu_v(\tau_v)^{-n}\pi_v(\diag(\tau_v\cdot 1_n, \tau_v^{-1}\cdot 1_n))w
        \end{align*}
        and
            \begin{align*}
        &Z_v(s,r(\delta_{\tau_v})\bfh_{v,s},w,w^\vee) \\
        &= \vol(K_v(\frakp_v^{m_v})) \cdot\gamma_v(-s-n+1/2, \mu_v^{-1}, \psi_{v,\bdelta_v^{-1}})^{-1}
        \prod_{j=1}^{n}\gamma_v(-2s-2n+2j, \mu_v^{-2}, \psi_{v,\bdelta_v^{-1}})^{-1}\\
        &\qquad\cdot \omega_{\pi_v}(-1)\gamma_v(-s+1/2, \pi_v^\vee, \mu_v^{-1}, \psi_{v,\bdelta_v^{-1}}) \mu_v(\tau_v)^{-n}(\pi_v(\diag(\tau_v\cdot 1_n, \tau_v^{-1}\cdot 1_n))w, w^\vee).
        \end{align*}
        \end{enumerate}
    \end{lem}
    \begin{proof}
        The functional equation shows that
        \[
        \mu_v(-1) \cdot \mathrm{Id} = M_{v,w_{2n},s}^\dagger \circ M_{v,w_{2n},-s}^\dagger =C(s) \frac{b_{v,2n}(-s,\mu_v^{-1})b_{v,2n}(s, \mu_v)}{a_{v,2n}(-s,\mu_v^{-1})}M_{v,w_{2n},s}^* \circ M_{v,w_{2n},-s}
        \]
        with some invertible function $C(s)$.
        Applying this to $\bff_{v}(-s,\mu_v^{-1})$ we obtain
        \[
        M_{v,w_{2n},s}^*\bfh_{v,s} = M_{v,w_{2n},s}^* \circ M_{v,w_{2n},-s}\bff_{v,-s} = \frac{\mu_v(-1)C(s)^{-1}a_{v,2n}(-s,\mu_v^{-1})}{b_{v,2n}(-s,\mu_v^{-1})b_{v,2n}(s,\mu_v)}\bff_{v}(-s,\mu_v^{-1}).
        \]
        This expression yields the first statement.
        Using the functional equation in Theorem \ref{th_functional_eq_zeta_int}, one has
        \begin{align*}
        &Z_v(s,\bfh_{v,s},w,w^\vee) 
        =
        Z_v(s, M_{v,w_{2n},-s}\bff(-s,\mu_v^{-1}), w,w^\vee)\\
        &=C(-s,\mu_v^{-1}, \psi_{v,\bdelta_v^{-1}})
         Z_v(s, M_{v,w_{2n},-s}^\dagger\bff(-s,\mu_v^{-1}), w,w^\vee)\\
         &=
         C(-s,\mu_v^{-1}, \psi_{v,\bdelta_v^{-1}})
         \omega_{\pi_v}(-1)\gamma_v(-s+1/2, \pi_v^\vee,\mu_v^{-1}, \psi_{v,\bdelta_v^{-1}})Z_v(-s, \bff_{v}(-s,\mu_v^{-1}), w,w^\vee).
        \end{align*}
The statements follow from Lemma \ref{lem_zeta_int_ramified_big_cell} and the definition of $Z_v(s,f_{v,s},w,w^\vee)$.
    \end{proof}

Related local computations for unitary groups at split $p$-adic places appear in \cite[Theorem 4.3.10]{2020_EHLS}. 
There, the authors use local Siegel-Weil sections and reduce the corresponding zeta integrals to Godement-Jacquet integrals. 
In contrast, our section $\bfh_{v,s}$ is constructed using the intertwining operator.
At the points in Lemma \ref{lem_ram_zeta_int}, it lies in the kernel of the intertwining operator and allows us to apply the Siegel-Weil formula in the proof of Lemma \ref{lem_holomorphy_Eis_ser}.

\subsection{Liu's result at archimedean places}

    We recall Liu's result \cite{2021_Zheng_Liu_arch_zeta_int} of the archimedean zeta integral for a suitable section of $I_{\infty, 2n}(s, \sgn^k)$.
    For $(k, \lambda) = (k,\lambda_1,\ldots,\lambda_n) \in \bbZ_{>0} \times \bbZ^n$ with $\lambda_1 \geq \cdots \geq \lambda_n \geq k > n$, consider a polynomial in the entries of $n \times n$ matrices
    \[
    \frakQ_{k,\lambda} = \prod_{j=1}^{n-1} {\det}_j^{\lambda_j - \lambda_{j+1}} \cdot {\det}^{\lambda_n-k},
    \]
    where $\det_j$ is the determinant of the upper left $j \times j$ block.
    Define an element of $\calU(\frakp_{2n,+})$ by
    \begin{equation}\label{def_diff_op}
    \frakQ_{k,\lambda} \left(\frac{1}{2\pi \sqrt{-1}}(X_{i+n,j})_{1 \leq i,j \leq n}\right)
    \end{equation}
    where $X_{i,j}$ is the fixed root vector with root $e_i + e_j$.
    This is well-defined since $\frakp_{2n,+}$ is commutative.
    Note that by multiplying by $2$, the definitions of the root vectors in \cite[\S A.8]{00_Shimura} and \cite{2020_Liu} are consistent.
    The present paper follows the definition in \cite{00_Shimura}, and the above definition yields the same element as in \cite{2020_Liu}.

    For an integer $k$, let $f_k$ be the weight $k$ vector in $I_{\infty, 2n}(k-(2n+1)/2, \sgn^k)$ defined by
    \[
    f_k\left(\begin{pmatrix}
        A&B\\
        C&D
    \end{pmatrix}\right)
    =
    \det(\sqrt{-1}\,C+D)^{-k}.
    \]
    We extend $f_k$ to the standard section of $I_{\infty,2n}(s,\sgn^k)$ and denote it by $f_{\infty,s,k}$.
    Let $f_{k,\lambda, s}$ be a standard section such that
    \[
    f_{k,\lambda, k-(2n+1)/2} = \frakQ_{k,\lambda} \left(\frac{1}{2\pi \sqrt{-1}}(X_{i+n,j})_{i,j}\right)f_{k}.
    \]
    Then, $f_{k,\lambda, k-(2n+1)/2}$ generates a lowest weight module with lowest weight $k$.
    Set
    \[
    \Gamma_0(s) = 1, \qquad \Gamma_m(s) = \pi^{m(m-1)/4}\prod_{j=0}^{m-1}\Gamma(s-j/2).
    \]
    Zheng Liu \cite{2021_Zheng_Liu_arch_zeta_int} computes the doubling integral for the section $f_{k,\lambda,s}$ at $s = k-n-1/2$.
    
    \begin{thm}[{\cite{2021_Zheng_Liu_arch_zeta_int}}]\label{thm_arch_zeta}
        Let $\calD_\lambda$ be a holomorphic discrete series representation of $G_n(\bbR)$ with lowest $K_\infty$-type $\lambda = (\lambda_1,\ldots,\lambda_n)$.
        Take a highest weight vector $v$ in the lowest $K$-type of $\calD_\lambda$ and its dual vector $v^\vee$ in the contragredient representation.
        We then have
        \[
        Z(k-n-1/2,f_{k,\lambda,s}, v, v^\vee)
        =
        \frac{(4\sqrt{-1}\ \pi)^{-\sum_{j=1}^n\lambda_j+nk}\cdot 2^{n(2(n+1)-3k)}\cdot \pi^{n(3n+1)/2}}{\Gamma_{2n}(k) \cdot \dim(\GL_n, \lambda)} \prod_{j=1}^n\Gamma(\lambda_j - j + k -n).
        \]
    \end{thm}

\section{Arithmeticity of critical \texorpdfstring{$L$}{}-values in the left half plane}

% In this section, we first fix the notation for automorphic representations and define the critical points.
% Then, we apply Lemma \ref{lem_ram_zeta_int} to construct an Eisenstein series which is a nearly holomorphic modular form.
% The rationality of the Eisenstein series follows directly from the explicit computation of its Fourier coefficients.
% Combining this rationality and the explicit computations of local zeta integrals in \S \ref{section_zeta_int}, we show the arithmeticity of the critical $L$-values.

\subsection{Critical points, choice of global sections and assumptions}\label{subsec_choice_global_section}
Throughout this section, we fix a cuspidal automorphic representation $\pi = \otimes_v \pi_v$ of $G_n(\bbA_F)$ such that $\pi_v$ is a holomorphic discrete series of weight $\lambda_v = (\lambda_{v,1},\ldots,\lambda_{v,n})$ for every archimedean place $v$.
Let $\mu$ be a Hecke character of $F^\times F_{\infty,+}^\times \bs \bbA_F^\times$.
For each archimedean place $v$, define $\epsilon_v, \epsilon_v' \in \{0,1\}$ by $\mu_v = \sgn^{\epsilon_v}$ and $\epsilon'_v \equiv \epsilon_v + n \bmod 2$.
The archimedean factor of the standard $L$-function $L(s,\pi,\mu)$ is given by
\[
    \Gamma_\bbR(s + \epsilon_v') \prod_{i=1}^n \Gamma_\bbC(s+\lambda_{v,i}-i),
\]
where $\Gamma_\bbR(s) = \pi^{-s/2}\Gamma(s/2)$ and $\Gamma_\bbC(s) = 2(2\pi)^{-s}\Gamma(s)$.
We say that $s=s_0 \in \bbZ$ is critical if $L_v(s, \pi_v, \mu_v)$ and $L_v(1-s, \pi_v^\vee, \mu_v^{-1})$ are holomorphic at $s=s_0$ for every archimedean places $v$ (see \cite[Definition 1.3]{79_Deligne_corvaris}).
    An $L$-value at a critical point is called a critical value.
    Critical points exist only if $\epsilon_v$ is independent of $v \mid \infty$.
    Letting $\vep \in \{0,1\}$ with $\mu_v = \sgn^\vep$, the set of critical points equals
    \[
    \{s_0, 1-s_0 \in \bbZ \mid 1 \leq s_0 \leq \min_v\lambda_{v,n}-n,\ s_0 \equiv \vep + n \bmod 2\}.
    \]
    Deligne \cite{79_Deligne_corvaris} formulated a conjecture concerning the arithmeticity of critical $L$-values.
    Conjecturally, critical $L$-values are algebraic up to certain periods.
    In our result, the period in Deligne’s conjecture is replaced in our formula by the Petersson norm of a Hecke eigenform together with some power of $2\pi\sqrt{-1}$, and we prove the arithmeticity of critical $L$-values in the left half plane.

Now we construct a global section.
Suppose that $L(s,\pi,\mu)$ has a critical point $s_0$.
Take the following data:
\begin{itemize}
    \item $(\tau_v)_v \in \widehat{\calO}_F^\times$;
    \item disjoint finite sets $S_1$ and $S_2$ of nonarchimedean places such that $S_1 \sqcup S_2$ contains all the ramified places;
    \item $m_v\in \bbZ_{>0}$ for any $v \in S_1 \sqcup S_2$ such that $\pi_v^{K_v(\frakp_v^{m_v})} \neq 0$ and $\mu_v|_{1+\frakp_v^{m_v}} = \mathbf{1}$.
\end{itemize}
For a nonarchimedean place $v \not\in S_1 \sqcup S_2$, put $m_v = 0$.
    Set
    \[
    s_0'
    =
    \begin{cases}
        s_0 & \text{if $s_0 > 0$;}\\
        1-s_0 & \text{if $s_0 \leq 0.$}
    \end{cases}
    \]
    Hence, $s_0' \geq 1$ and it is a critical point.
    We define a meromorphic section $f_s = \otimes_v f_{v,s}$ of $I_{2n}(s,\mu) = \otimes_v I_{2n,v}(s,\mu_v)$ as follows.
    \begin{itemize}
        \item for an archimedean place $v$, $f_{v,s}$ is the standard section such that $f_{v,s} = f_{s_0' +n, \lambda_v, s}$;
        \item for $v \not \in S_1 \sqcup S_2$ with $v < \infty$, $f_{v,s}$ is the unramified section with $f_{v,s}(1_{4n})=1$;
        \item for $v \in S_1$, $f_{v,s} = 
        \mu_v(\tau_v)^{2n} \cdot r(\delta_{\tau_v})\bfh_{v,s}$;
        \item for $v \in S_2$, $f_{v,s} =  r(\delta_{\tau_v})\bff_{v,s}$.
    \end{itemize}
    The section $f_s$ depends on $(m_v)_v, (\tau_v)_v, (\lambda_v)_v$ and $S_1,S_2$ rather than $\pi_v$.
    Let $S_\infty$ be the set of archimedean places and $S = S_1 \sqcup S_2 \sqcup S_\infty$.
    We assume $S_1 \neq \emptyset$ if $s_0 \in \{0,1\}, F=\bbQ$ and $\mu^2 = \mathbf{1}$.

\subsection{Eisenstein series and Fourier coefficients}
    For a holomorphic section $h_s$ of $I_{2n}(s,\mu)$, we define the Siegel Eisenstein series $E(g,s,h)$ by
    \[
    E(g,s,h) = \sum_{\gamma \in P_{2n}(F) \bs G_{2n}(F)} h_s(\gamma g), \qquad g \in G_{2n}(\bbA_F).
    \]
    This series converges absolutely for $\mathrm{Re}(s)>n+1/2$ and admits a meromorphic continuation to all $s$.
    Below, with our choice of $f_s$, we consider the associated Eisenstein series $E(g,s,f)$.
    When we emphasize that $E(g,s,f)$ depends on the choice of $(\lambda_v)_v, \mu$ and $(\tau_v)_v$, we denote it by $E(g,s,f,\mu,(\lambda_v)_v, (\tau_v)_v)$.
    
    \begin{lem}\label{lem_holomorphy_Eis_ser}
        The Eisenstein series $E(g,s,f)$ is holomorphic at $s = s_0'-1/2 = |s_0-1/2|$ and $E(g,s_0'-1/2,f)$ corresponds to a nearly holomorphic modular form.
        Furthermore, if $(\lambda_v)_v$ is parallel scalar weight $s_0'+n$, then
        $E(g,s_0'-1/2,f)$ corresponds to a holomorphic modular form.
    \end{lem}
    \begin{proof}
    We only prove the case $s_0' = 1, F=\bbQ$ and $\mu^2 = \mathbf{1}$.
    The other cases are similar and easier (cf.~\cite[Lemma 4.3.1]{Horinaga_3}).
    For $p \in S_1$, the section $f_{p,1/2}$ lies in the kernel of $M_{p,w_{2n},1/2}^*$ by Lemma \ref{lem_ram_zeta_int}.
    For a quadratic space $V_v$ of $\bbQ_v$, let $R(V_v)$ be the image of the theta correspondence of the trivial representation of the orthogonal group $\O(V_v)$.
    By \cite[Proposition 5.6]{1992_Kudla-Rallis}, $f_{p,1/2} \in R(V_p^+) \cap R(V_p^-)$, where $V_p^\pm$ is the $2n+2$-dimensional quadratic spaces $V$ over $\bbQ_p$ with Hasse invariants $\pm 1$ and the associated quadratic character $\mu_p$.
    Let $V(2n+2,0)$ denote the positive definite $2n+2$-dimensional quadratic space over $\bbR$.
    The archimedean component $f_{\infty, 1/2}$ lies in $R(V(2n+2,0))$.
    We may decompose $f_{1/2}$ as a sum of sections so that $f_{1/2} \in \sum_{\{V_v\}_v} \otimes_v R(V_v)$, where $V_v$ runs over quadratic spaces over $\bbQ_p$ of dimension $2n+2$ such that $V_\infty = R(2n+2,0)$.
    We may choose $V_p$ as $V_p^\pm$ appropriately.
    Each family $\{V_v\}_v$ gives rise to a global quadratic space $V$.
    Hence, the section $f_{1/2}$ is a sum of Siegel-Weil sections in the sense of \cite{1988_Kudla-Rallis}.
    %, i.e., a sum of sections associated with Schwartz functions $\calS(V(\bbA))$ for quadratic spaces $V$ over $\bbQ$ so that $V \otimes \bbR = V(2n+2,0)$ and $V \otimes \bbQ_p = V_p^\pm$.
    Let $f_s^{\std}$ be the standard section such that $f_{1/2}^{\std} = f_{1/2}$.
    Since the quadratic space $V$ is anisotropic, the resulting Eisenstein series for $f_s^{\std}$ is nearly holomorphic at $s=1/2$ (cf.~\cite[Lemma 4.3.1]{Horinaga_3}).
    For $h_s = (s-1/2)^{-1}(f_s - f_s^\std)$,
    we have
    \[
    E(g,s,f)|_{s=1/2} = E(g,1/2,f^\std) + \mathrm{Res}_{s=1/2} E(g,s,h_s).
    \]
    Since the archimedean section of $f_s$ is standard, we may assume $h_{\infty,s} = f_{\infty,s}$.
    Again, take the standard section $h_{s}^\std$ such that $h_{1/2}^\std = h_{1/2}$.
    At $s=1/2$, the section $h_{s}^\std$ is a sum of sections associated with an incoherent family $\{V_v\}_v$ and Siegel-Weil sections associated with positive definite quadratic spaces.
    By \cite[Theorem 3.3]{2008_Kudla}, the residue of $E(g,s,h_s^\std)$ is zero at $s=1/2$.
    Since the pole of the Siegel Eisenstein series is at most simple, we have $E(g,s,f)|_{s=1/2} = E(g,1/2,f^\std)$.
    Hence, at $s=1/2$, the Eisenstein series $E(g,s,f)$ is holomorphic and corresponds to a nearly holomorphic modular form.
    
    If $(\lambda_v)_v$ has parallel weight $s_0'+n$, then each archimedean section generates the lowest weight representation of weight $s_0'+n$ and is a lowest weight vector in that representation.
    Hence, the Eisenstein series corresponds to a holomorphic modular form.
    \end{proof}
    
    We next compute the local factors appearing in the Fourier coefficients: the Siegel series and the confluent hypergeometric functions.
    Let $h \in \Sym_{2n}(F)$ be a totally positive definite symmetric matrix and $h_v$ be the element in $\Sym_{2n}(F_v)$ corresponding to $h$.
    Put $B_v = \bdelta_vh_v$ and
        \[
        B'_v = \begin{pmatrix}
            0_n&1_n\\
            1_n&0_n
        \end{pmatrix}B_v.
        \]
    For a fractional ideal $\frakA$ of $\calO_v$, let $\Sym_{2n}(\frakA)^*$ be the dual of $\Sym_{2n}(\frakA)$ with respect to the bilinear form $(x,y) \mapsto \tr(xy)$.
    
    For $B \in \Sym_{2n}(\calO_v)^*$ with $\det B \neq 0$, we define a polynomial $\gamma(B,X)$ with coefficients in $\bbZ$ by
    \[
    \gamma(B,X) = \frac{(1-X)\prod_{j=1}^n(1-q_v^{2j}X^2)}{1-q^n_v\xi_B X}
    \]
    where $\xi_B$ is $+1$ (resp.\,$-1$) if $D_B \in F_v^{\times2}$ (resp.\,$F_v(\sqrt{D_B})/F_v$ is a quadratic unramified extension), and zero otherwise.
    Let $f_{v,s}^0$ be the unramified standard section of $I_{v,2n}(s,\mathbf{1})$ with $f_{v,s}^0(1_{4n}) = 1$.
    Set
    \[
    b(B,s) = \int_{\Sym_{2n}(F_v)}
        f^0_{v,s-n-1/2}(w_{2n}\bfn(x)) \overline{\psi_{v}(\tr(Bx))}
        \ dx.
    \]
        
    \begin{thm}[{\cite[Theorem 16.2]{00_Shimura}}]\label{thm_FC_unram}
        There exists a polynomial $F(B,X) \in \bbZ[X]$ such that
        \[
        b(B,s) = \gamma(B,q_v^{-s})F(B,q_v^{-s}).
        \]
    \end{thm}
    For $v \in S_1 \sqcup S_2$, we can explicitly compute our Siegel series.
    \begin{lem}\label{lem_FC_ram}
    For $v \in S_1 \sqcup S_2$, we have 
    \[
    \int_{\Sym_{2n}(F_v)}f_{v,s}(w_{2n}\bfn(x)) \overline{\psi_{v}(\tr(hx))} \ dx = 0
    \]
    if $B_v \not\in \frakp_v^{-m_v}\Sym_{2n}(\calO_v)^*$.
    If $B_v \in \frakp_v^{-m_v}\Sym_{2n}(\calO_v)^*$, we have
    \begin{align*}
    &\int_{\Sym_{2n}(F_v)}f_{v,s}(w_{2n}\bfn(x)) \overline{\psi_{v}(\tr(hx))} \ dx\\
    &=
    C_v
    \psi_{v, \bdelta_v^{-1}}(\tr(B_v'\tau_v^{-1}))
    \cdot
    \begin{dcases}
    \mu_v(\det(B_v))|\det(B_v)|_v^{s}
        c(-s,\mu_v^{-1},B_v,\psi_{v,\bdelta_v^{-1}}^{-1})
         & \text{if $v \in S_1$;}\\
        \mu_v(\tau_v)^{2n} & \text{if $v \in S_2$,}
    \end{dcases}
    \end{align*}
    with the rational constant $C_v = \vol(\Sym_{2n}(\frakp_v^{m_v}))\mu_v((-1)^n)$.
    \end{lem}
    \begin{proof}
    When $v \in S_1$, by the functional equation of the Siegel series \eqref{FE_Siegel_series}, we have
        \begin{align*}
        &\int_{\Sym_{2n}(F_v)}
        f_{v,s}
        \left(w_{2n}\bfn(x)\right)
        \psi_{v, \bdelta_v^{-1}}(-\tr(B_v x))\ dx 
        = 
        \ell_{\psi_{v, \bdelta_v^{-1}}^{-1}, B_v} (f_{v,s}) \\
        &= \ell_{\psi_{v, \bdelta_v^{-1}}^{-1}, B_v} (M_{v,w_{2n},-s}
        (\mu_v(\tau_v)^{2n}
        r(\delta_{\tau_v})\bff_{v}(-s,\mu_v^{-1})
        ))\\
        &=
        \mu_v(\tau_v^{2n}\det(B_v))
        |\det(B_v)|_v^{s} 
        c(-s,\mu_v^{-1},B_v,\psi_{v,\bdelta_v^{-1}}^{-1})
        \cdot
        \ell_{
        \psi_{v, \bdelta_v^{-1}}^{-1}, B_v}
        (r(\bdelta_{\tau_v})\bff_{v}(-s,\mu_v^{-1})).
        \end{align*}
    Hence, it suffices to consider the case where $v \in S_2$.
        When $v \in S_2$, for $x \in \Sym_{2n}(F_v)$, $f_{v,s}(w_{2n}\bfn(x))$ is nonzero only if $w_{2n}\bfn(x)\delta_{\tau_v} \in P_{2n}K_v(\frakp_v^{m_v})$, which is equivalent to $x \in \GL_{2n}(\calO_{v})$ and
        \[
        x^{-1} \equiv \tau_v\begin{pmatrix}
            0_{n} & -1_n\\
            -1_n & 0_n
        \end{pmatrix}
        \bmod \frakp_v^{m_v}
        \]
        by a computation similar to that in Lemma \ref{lem_zeta_int_ramified_big_cell}.
        Conversely, for such $x$, $f_{v,s}(w_{2n}\bfn(x))$ is nonzero and is equal to $\mu_v((-1)^n\tau_v^{2n})$.
        Hence we obtain
        \begin{align*}
        &\int_{\Sym_{2n}(F_v)}f_{v,s}
        \left(w_{2n}\bfn(x)\right) 
        \psi_{v, \bdelta_v^{-1}}(-\tr(B_v x))\ dx\\
        &=
        \mu_v((-1)^n\tau_v^{2n})
        \int_{
        \tau_v^{-1}\left(\begin{smallmatrix}
            0_n & -1_n\\
            -1_n & 0_n
        \end{smallmatrix}\right) + \Sym_{2n}(\frakp_v^{m_v})} 
        \psi_{v, \bdelta_v^{-1}}\left(-\tr\left(B_v x\right)\right)\ dx\\
        &=
        \mu_v((-1)^n\tau_v^{2n})
        \psi_{v, \bdelta_v^{-1}}(\tr(B_v'\tau_v^{-1}))
        \cdot
        \int_{\Sym_{2n}(\frakp_v^{m_v})} \psi_{v, \bdelta_v^{-1}}\left(-\tr\left(B_v x\right)\right)\ dx\\
        &=
        \mu_v((-1)^n\tau_v^{2n})
        \cdot
        \begin{cases}
            \vol(\Sym_{2n}(\frakp_v^{m_v}))\psi_{v, \bdelta_v^{-1}}(\tr(B_v'\tau_v^{-1})) & \text{if $B_v \in \frakp_v^{-m_v}\Sym_{2n}(\calO_v)^*$;}\\
            0 & \text{otherwise.}
        \end{cases}
        \end{align*}
        This completes the proof.
    \end{proof}
    For archimedean places, we will consider the section $f_{\infty,s,k}$.
    Let $\xi(g,h;s,s')$ be the confluent hypergeometric function defined by
    \[
    \xi(g,h;s,s')
    =
    \int_{\Sym_{2n}(\bbR)} \det(x+\sqrt{-1}\,g)^{-s} \det(x-\sqrt{-1}\,g)^{-s'} \exp(-2\pi \sqrt{-1}\,\tr(hx)) \ dx.
    \]
    Here for $s,s' \in \bbC$, we choose the branches of $\det(z)^s$ and $\overline{\det(z)}^{s'}$ so that their values at $z = \sqrt{-1} \cdot 1_{2n}$ are $\exp(\pi\sqrt{-1}\,ns)$ and $\exp(-\pi\sqrt{-1}\,ns')$, respectively.
    For $z = x + \sqrt{-1}\,y \in \frakH_{2n}$, set
    \[
    g_z = \begin{pmatrix}
            1_{2n} & x\\
            0_{2n} & 1_{2n}
        \end{pmatrix}
        \begin{pmatrix}
            y^{1/2}&0_{2n}\\
            0_{2n}&y^{-1/2}
        \end{pmatrix}.
    \]
    We compute the degenerate Whittaker integral for $f_{\infty,s,k}$.
    \begin{lem}\label{lem_FC_arch}
    For an archimedean place $\infty$, we have
    \begin{align*}
        &\int_{\Sym_{2n}(\bbR)} f_{\infty,s,k}(w_{2n} \bfn(u)g_z) \overline{\psi_\infty(\tr(h_\infty u))} \ du\\
        &=
        \exp(2\pi\sqrt{-1}\,\tr(h_\infty x))(\det(y))^{(s+n+1/2)/2}\xi(y,h_\infty;(s+k+n+1/2)/2, (s-k+n+1/2)/2).
        \end{align*}
        Substituting $s=k-n-1/2$, it equals
        \[
        2^{-2n^2+n}(-2\pi\sqrt{-1})^{2nk}\Gamma_{2n}(k)^{-1}(\det(h_\infty))^{k-n-1/2} (\det(y))^{k/2}\exp(2\pi\sqrt{-1}\,\tr(h_\infty z)).
        \]
    \end{lem}
    \begin{proof}
        The first statement follows from a straightforward computation using the Iwasawa decomposition of $w_{2n}\bfn(u)g_z$.
        The last statement follows from the first statement and \cite[(7.12)]{83_Shimura_On_Eis_ser}.
    \end{proof}
With the above preliminaries, we study the Fourier coefficients of Eisenstein series of parallel weight $k$ at $s = k - n -1/2$.
For $z = (z_v)_v = (x_v + \sqrt{-1}\, y_v)_v \in \prod_{v \mid \infty} \frakH_{2n}$, let $g_z$ be the element in $G_{2n}(\bbA_F)$ such that the finite part of $g_z$ is $1$ and $g_v = g_{z_v}$ for $v \mid \infty$.
    We consider the Fourier coefficients of the Siegel Eisenstein $E(g_z, s, f)$.
    When the Eisenstein series has parallel weight $k$, it can be expressed as
    \[
        E(z, s, f) =
        \prod_{v \mid \infty}
        \det(y_{v})^{-k/2} \cdot E(g_z, s, f).
    \]
    
    \begin{prop}\label{prop_FC_Eis_ser}
    Suppose that $\lambda_v = (k,\ldots,k)$ with $k = s_0' + n$.
    For a totally positive definite $h \in \Sym_{2n}(F)$, the $h$-th Fourier coefficient of $E(g_z, s_0'-1/2, f)$ is zero if $h_v \not\in \frakd_v^{-1} \frakp_v^{-m_v} \Sym_{2n}(\calO_v)^*$ for some $v < \infty$.
    If $h_v\in \frakd_v^{-1} \frakp_v^{-m_v} \Sym_{2n}(\calO_v)^*$ for every $v < \infty$, the $h$-th Fourier coefficient is given by
    \begin{align*}
        &C(h)D_F^{-n(2n+1)/2}\cdot 
        \frac{L^S(k-n, \mu\chi_{D_h})}{L^S(k, \mu) \prod_{j=1}^n L^S(2k-2j, \mu^2)} \prod_{v \in \bfc} F(B_v, \mu_v(\varpi_v)q^{-k}_v)\\
        &\qquad\cdot \prod_{v \in S_1} 
        \mu_v(\tau_v)^{-2n}
        c(n-k+1/2 ,\mu_v^{-1},B_v,\psi_{v,\bdelta_v^{-1}}^{-1})\mu_v(\det(B_v))|\det(B_v)|_v^{-1/2} \\
        &\qquad \qquad\cdot \prod_{v \in S_1 \sqcup S_2}
        \mu_v(\tau_v)^{2n}
        \psi_{v,\bdelta_v^{-1}}(\tr(B_v'\tau_v^{-1}))\cdot
        \prod_{v \mid \infty}(-2\pi\sqrt{-1})^{2nk}\Gamma_{2n}(k)^{-1}
        \det\left(h_{v}\right)^{-1/2}.
        \end{align*}
        Here, $C(h)$ is a rational constant defined by
        \begin{align*}
        C(h) 
        &=
        2^{d(-2n^2+n)} \cdot
        \prod_{v \in S_1 \sqcup S_2}
        \vol(\Sym_{2n}(\frakp_v^{m_v}))\mu_v(-1)^n
        \cdot 
        \prod_{v \in S_1} |\det(B_v)|^{k-n}_v
        \cdot 
        \prod_{v \mid \infty} \det(h_v)^{k-n}
        \end{align*}
        and $\bfc$ denotes the finite set of unramified places for which $F(B_v, X)$ is not constant. 
\end{prop}
\begin{proof}
    Unfolding the Eisenstein series shows that the $h$-th Fourier coefficient of $E(g_z, s,f)$ equals
        \begin{align*}
        &\int_{\Sym_{2n}(\bbA_F)} f_{s_0'-1/2}(w_{2n} \bfn(x)g_z) \overline{\psi_F(\tr(hx))} \ dx \\
        &= 
        D_F^{-n(2n+1)/2}\prod_{v} \int_{\Sym_{2n}(F_v)} f_{v,s_0'-1/2}(w_{2n}\bfn(x)g_z) \overline{\psi_{v}(\tr(h_v x))} \ dx.
        \end{align*}
        Consider right translation by $u = (u_v)_v \in \prod_{v < \infty}N_{2n}(\frakp_{v}^{m_v})$.
        By the definition of Eisenstein series, the Eisenstein series is invariant under right translation by $u$.
        Hence, if the $h$-th Fourier coefficient is nonzero, we have $\tr(h_v x_v) \in \frakd_v^{-1}$ for every $x_v \in \Sym_{2n}(\frakp_v^{m_v})$, equivalently, $h_v \in \frakd_v^{-1}\frakp_v^{-m_v} \Sym_{2n}(\calO_v)^*$ for any $v<\infty$.
        When $h_v \in \frakd_v^{-1}\frakp_v^{-m_v} \Sym_{2n}(\calO_v)^*$ for any $v < \infty$,
        combining Theorem \ref{thm_FC_unram}, Lemmas \ref{lem_FC_ram} and \ref{lem_FC_arch}, we obtain the statement.
\end{proof}

\subsection{Arithmeticity of Eisenstein series}
    In this subsection, we prove the rationality of Eisenstein series.
    For a Hecke character $\chi$ of $F^\times F_{\infty,+}^\times \bs \bbA_F^\times$, let $G(\chi)$ be the Gauss sum defined by $G(\chi) = D_F^{-1/2} \prod_{v < \infty} \vep_v(0,\chi_v,\psi_v)$.
    For $\sigma \in \Aut(\bbC)$, let $\tau_{\sigma,v}$ be the element in $\calO_v^\times$ defined in \S \ref{subsec_Gal_act_on_AR}. 

\begin{lem}\label{lem_arith_Gauss_sum_L-val}
    Let $\chi$ be a Hecke character with parallel signature, and $m \geq 1$ be a critical point of $L(s,\chi)$.
    For $\sigma \in \Aut(\bbC)$, with the above notation, the following statements hold.
    \begin{enumerate}
        \item $\sigma(G(\chi)) = {^\sigma \chi}(\tau_\sigma) G({^\sigma \chi})$.
        \item For a Hecke character $\chi'$, we have
        \[
        \sigma\left(\frac{G(\chi\chi')}{G(\chi)G(\chi')}\right) = \frac{G({^\sigma \chi} {^\sigma \chi'})}{G({^\sigma \chi})G({^\sigma \chi'})}.
        \]
        \item We have 
        \[
        \sigma
        \left(
        \frac{L^S(m,\chi)}{D_F^{1/2}G(\chi)(2\pi \sqrt{-1})^{dm}}
        \right)
        = 
        \frac{L^S(m,{^\sigma \chi})}{D_F^{1/2}G({^\sigma \chi})(2\pi \sqrt{-1})^{dm}}.
        \]
        \item We have
    \[
    \sigma
    \left(
    \frac{\sqrt{-1}^{nd}L^S(k-n, \mu\chi_{D_h})}
    {D_F^{1/2}G(\mu)(2\pi\sqrt{-1})^{(k-n)d}} \cdot  \prod_{v \mid \infty} \det(h_v)^{-1/2}
    \right) 
    = 
    \frac{\sqrt{-1}^{nd}L^S(k-n, {^\sigma \mu}\chi_{D_h})}{D_F^{1/2}G({^\sigma \mu})(2\pi\sqrt{-1})^{(k-n)d}} \cdot  \prod_{v \mid \infty} \det(h_v)^{-1/2}
    \]
    and $q_v^{-\frac{1}{2}c(\chi_{D_{B_v}})} |\det(B_v)|^{-1/2}_v \in \bbQ$.
    \item We have 
    \[
    \sigma\left(
    \frac{c(n-k+1/2, \mu_v^{-1}, B_v, \psi_{v,\bdelta_v^{-1}})}{\mu_v(\tau_v)^{2n}|\det(B_v)|_v^{1/2}}
    \right)
    =
    \frac{c(n-k+1/2, {^\sigma}\mu_v^{-1}, B_v, \psi_{v,\bdelta_v^{-1}})}{\mu_v(\tau_v\tau_{\sigma,v}^{-1})^{2n}|\det(B_v)|^{1/2}_v}
    \]
    \end{enumerate}
\end{lem}
\begin{proof}   
    The first statement follows from the definition of the Gauss sum \eqref{def_epsilon_factor} (see also \cite[Lemma A.3]{22_Shih-yu_Chen}), and the second is immediate.
    For the third, see \cite[Theorem 18.12]{00_Shimura} and \cite[Theorem 3.3(4)]{2023_Chen_Shih-yu_Advance_sym_fourth}.
    For the fourth statement, we first show $q_v^{-\frac{1}{2}c(\chi_{D_{B_v}})}|\det (B_v)|_v^{-1/2} \in \bbQ$.
    If $D_{B_v} \in F_v^{\times 2}$, there is nothing to prove.
    If $D_{B_v} \not\in F_v^{\times 2}$, let $\Delta$ be the discriminant ideal with respect to the quadratic extension $F_v(\sqrt{D_{B_v}})/F_v$.
    By Corollary 2 of \cite[p.240]{1967_Weil_BNT}, the exponent of $\Delta$ is $c(\chi_{D_{B_v}})$.
    There exists $x \in F_v^\times$ such that $\Delta = 4D_{B_v}x^2 \calO_v$.
    Taking the valuation of both sides, we have $\ord_{v}(D_{B_v}) \equiv c(\chi_{D_{B_v}}) \bmod 2$, which shows the claim.
    By the above discussion, it remains to see
    \[
    \sigma(\sqrt{-1}^{nd}G(\chi_{D_h})|N_{F/\bbQ}(\det h)|^{-1/2}) = \sqrt{-1}^{nd} G(\chi_{D_h})|N_{F/\bbQ}(\det h)|^{-1/2}.
    \]
    Since $h$ is totally positive, we have $G(\chi_{D_h})^2 = (-1)^{nd}\prod_{v < \infty} q_v^{c_v'}$, where $c_v'$ is the conductor exponent of $\chi_{D_h,v}$.
    By $\prod_{v \leq \infty} |\det h|_v = 1$ and $q_v^{-\frac{1}{2}c(\chi_{D_{B_v}})}|\det(B_v)|_v^{-1/2} \in \bbQ$, the assertion holds.
    For the last statement, let $m$ be an integer and $\chi_v$ be a character of $F_v^\times$.
    We then have
    \[
    \sigma(\vep_v(0, \chi_v, \psi_{v,\bdelta_v^{-1}})^{-1}\gamma_v(m,\chi_v, \psi_{v,\bdelta_v^{-1}}))
    =
    \vep_v(0, {^\sigma}\chi_v, \psi_{v,\bdelta_v^{-1}})^{-1}\gamma_v(m,{^\sigma}\chi_v, \psi_{v,\bdelta_v^{-1}})
    \]
    Hence, the following expression is Galois equivariant:
    \begin{align*}
        \frac{
        \vep_v(0,\mu_v^{-1}, \psi_{v,\bdelta_v^{-1}})
        \vep_v(0,\mu_v^{-2}, \psi_{v,\bdelta_v^{-1}})^n
        \vep_v(1/2, \chi_{D_B}, \psi_{v,\bdelta_v^{-1}})}{\vep_v(0,\mu_v^{-1}\chi_{D_B}, \psi_{v,\bdelta_v^{-1}})}
        c(n-k+1/2, \mu_v^{-1},B_v, \psi_{v,\bdelta_v^{-1}}).
    \end{align*}
    By the relation $\vep_v(1/2,\chi_{D_{B_v}},\psi_{v,\bdelta_v^{-1}}) = q_v^{-c(\chi_{D_{B_v}})/2} \vep_v(0,\chi_{D_{B_v}}, \psi_{v,\bdelta_v^{-1}})$ and the preceding assertion in Lemma \ref{lem_arith_Gauss_sum_L-val} (4) together with the local analogue of Lemma \ref{lem_arith_Gauss_sum_L-val} (1), (2), we obtain the result.
\end{proof}

        \begin{prop}\label{prop_rationality_Eis_ser}
    Let $f_s$ be the section defined in \S \ref{subsec_choice_global_section}.
    Then, for any $\sigma \in \Aut(\bbC)$, we have 
    \[
    {^\sigma E}(g, s_0'-1/2, f, \mu, (\lambda_v)_v,  (\tau_v)_v)
    =
    E(g, s_0'-1/2,f, {^\sigma \mu}, (\lambda_{\sigma \circ v})_v, (\tau_v\tau_{\sigma,v}^{-1})_v).
    \]
    \end{prop}
    \begin{proof}
        We first treat the case where $\lambda = (k,\ldots,k)$, for which $s_0'-1/2 = k-n-1/2$.
        By Lemma \ref{lem_holomorphy_Eis_ser}, the Eisenstein series is a holomorphic modular form.
        Let $h \in \Sym_n(F)$ be a totally positive definite matrix and $a(h,\mu, (\lambda_v)_v, (\tau_v)_v)$ be the $h$-th Fourier coefficient of the Eisenstein series $E(g_z, k-n-1/2, f, \mu, (\lambda_v)_v, (\tau_v)_v)$.
        By Proposition \ref{prop_FC_Eis_ser}, Lemma \ref{lem_arith_Gauss_sum_L-val} and $\Gamma_{2n}(k) \in \pi^{n^2} \bbQ^\times$, we have 
        $\sigma(a(h,\mu,(\lambda_v)_v, (\tau_v)_v) = a(h, {^\sigma \mu},(^\sigma\lambda_v)_v, (\tau_v\tau_{\sigma,v}^{-1})_v)$
        for any $\sigma \in \Aut(\bbC)$.
        Hence, the $h$-th Fourier coefficient of the difference
        \[
        {^\sigma E}(g, s_0'-1/2, f, \mu, (\lambda_v)_v,  (\tau_v)_v)
    -
    E(g, s_0'-1/2,f, {^\sigma \mu}, (\lambda_{\sigma \circ v})_v, (\tau_v\tau_{\sigma,v}^{-1})_v)
        \]
        is zero.
        Such a modular form is singular and, by \cite[Theorem 5.6]{94_Shimura}, has weight at most $n$.
        Since the Eisenstein series is of weight $k \geq n + 1$, the difference is zero.
        Hence, the statement holds when the weight is parallel scalar weight $(k,\ldots,k)$ with $k = s_0' + n$.

        For general weights, by Lemma \ref{lem_arith_AF}, the differential operator \eqref{def_diff_op} respects the Galois action.
        This proves the statement.
    \end{proof}
    
\subsection{Integral representation}

    A fundamental feature of the standard $L$-functions is their integral representation.
    We evaluate this integral representation in two different ways: using the zeta integral, and using the restriction to $G_n(\bbA_F) \times G_n(\bbA_F)$.
    Comparing the two expressions, we deduce the arithmeticity of the critical values.
    We first compute the integral representation.
    Let $\langle \,,\,\rangle$ be the Petersson inner product with respect to $dg^\Tam$.

    \begin{thm}[{\cite{87_Garrett, GPSR_87, 00_Shimura}}]\label{thm_int_rep_global_L}
        Let $h_s = \otimes_v h_{v,s}$ be a meromorphic section of $I_{2n}(s,\mu)$, and let $T$ be a finite set of places such that for every place $v \not\in T$, $h_{v,s}, \pi_v$ and $\mu_v$ are unramified.
        Assume that $h_{v,s}(1_{4n}) = 1$ for $v \not\in T$.
        For a cusp form $\varphi$,  we have
        \[
        \langle E(\iota(-,y), s, h),  \overline{\varphi}\rangle =
        \frac{D_F^{-n(2n+1)/2}\pi^{dn(n+1)/2}}
{\prod_{j=1}^{n}\Gamma(j)^d\zeta_F(2j)} 
    \cdot\frac{L^{T}(s+1/2,\pi,\mu)}{b_{2n}^T(s,\mu)} \left(\prod_{v \in T} T_{v}(s, h_{v,s})\varphi\right)(\check{y}).
        \]
    \end{thm}
    \begin{proof}
        By the basic identity in \cite{GPSR_87} (see also \cite[Theorem 3.6]{21_PSS_algebraicity}), we have
        \[
        \int_{G_n(F) \bs G_n(\bbA_F)} E(\iota(x,y), s, h) \varphi(x) \, dx^\Tam
        =
        \int_{G_n(\bbA_F)} h_s(\delta^{-1}\iota(x,1_{2n})) \varphi(\check{y}x) \, dx^\Tam.
        \]
        The description of the Tamagawa measure given in \eqref{Tamagawa_measure_comparison} and the computations at unramified places discussed in Theorem \ref{thm_unr_zeta_int} imply the formula.
    \end{proof}

    Substituting our section $f_s$ into this formula, we obtain the following.
    \begin{prop}\label{prop_global_zeta_int_explicit}
    With the choices in \S \ref{subsec_choice_global_section}, we have
        \begin{align*}
        &
        \langle E(\iota(-,g), s_0'-1/2,f), \varphi\rangle\\
        &=
        C \cdot
        \frac{D_F^{-n(2n+1)/2}\pi^{dn(n+1)/2}}
{\prod_{j=1}^{n}\Gamma(j)^d\zeta_F(2j)} 
    \cdot
        \frac{L^S(s_0', \pi, \mu)}{b^S_{2n}(s_0'-1/2,\mu)}\cdot \prod_{v \in S_1}\gamma_v(1-s_0', \pi_v^\vee, \mu_v^{-1}, \psi_{v,\bdelta_v^{-1}}) 
        \cdot
        \prod_{v \in S_1 \sqcup S_2} \mu_v(\tau_v)^{n}\\
        & \quad \cdot \prod_{v \in S_1}
        \gamma_v(-s_0'-n+1, \mu_v^{-1},\psi_{v,\bdelta_v^{-1}})^{-1}\prod_{j=1}^{n}\gamma_{v}(-2s_0' - 2n + 2j+1, \mu_v^{-2}, \psi_{v,\bdelta_v^{-1}})^{-1}
        \\
        &\quad \cdot \prod_{v \mid \infty} (\sqrt{-1}\,\pi)^{-\sum_{j=1}^n \lambda_{v,j} + n(s_0'+n)}
        \pi^{n(n+1)/2} \cdot \prod_{v \in S_1 \sqcup S_2} 
        r(\diag(\tau_v \cdot 1_n, \tau_v^{-1}\cdot 1_n))
        \overline{\varphi(\check{g})}
        \end{align*}
        where the nonzero rational constant $C$ is given by
        \[
        C 
        = \prod_{v \in S_1}
        \omega_{\pi,v}(-1)\prod_{v \in S_1 \sqcup S_2} \vol(K_v(\frakp_v^{m_v})) \cdot \prod_{v \mid \infty} 
        \frac{2^{-2\sum_{j}\lambda_{v,j}-ns_0'+n^2+2n}\prod_{j=1}^n
        \Gamma(\lambda_{v,j}-j+s_0')}
        {(\Gamma_{2n}(s_0'+n)/\pi^{n^2})\dim(\GL_n, \lambda_v)}.
        \]
    \end{prop}
    \begin{proof}
        Substitute Theorem \ref{thm_unr_zeta_int}, Lemmas \ref{lem_zeta_int_ramified_big_cell}, \ref{lem_ram_zeta_int}, and Theorem \ref{thm_arch_zeta} into Theorem \ref{thm_int_rep_global_L}.
    \end{proof}
\subsection{Arithmeticity of critical \texorpdfstring{$L$}{}-values}

Before the proof of the main result, we investigate the restriction of Eisenstein series to the doubling embedding $G_n(\bbA_F) \times G_n(\bbA_F) \subset G_{2n}(\bbA_F)$.
    Let $\rho_\lambda$ be the irreducible representation of $\prod_{v \mid \infty} K_{v}$ with highest weight $(\lambda_v)_v$ and $K(\frakN) = \prod_{v < \infty} K_v(\frakp_v^{m_v})$ be the compact open subgroup of $G_n(\bbA_{F,\fini})$, where $\frakN = \prod_v\frakp_v^{m_v}$.
    Let $\calN_\cusp(\lambda, \frakN)$ be the subspace of $\calN(G_n)^{K(\frakN)}_{\rho_\lambda}$ consisting of cusp forms of weight $(\lambda_v)_v$ and $\calN_\Eis(\lambda, \frakN)$ be the space of automorphic forms orthogonal to $\calN_\cusp(\lambda, \frakN)$ in $\calN(G_n)^{K(\frakN)}_{\rho_\lambda}$.
    For a cuspidal automorphic representation $\xi$ of $G_n(\bbA_F)$, let $\calN_\cusp(\lambda, \frakN)[\xi]$ be the $\xi$-isotypic component of $\calN_\cusp(\lambda, \frakN)$.
    Take an orthogonal basis $\calB(\lambda, \frakN)[\xi]$ of $\calN_\cusp(\lambda, \frakN)[\xi]$ and a basis $\calB_\Eis(\lambda, \frakN)$ of $\calN_\Eis(\lambda, \frakN)$.

The following lemma plays an essential role in the study of the restriction of $E(g,s_0'-1/2,f)$.
\begin{lem}\label{lem_cpx_conj}
    With the above notation, for the complex conjugation $c$, we have $\overline{\varphi}(\check{h}) = r(w_n){^c\varphi}(h)$.
\end{lem}
\begin{proof}
    For a Hilbert-Siegel modular form $f$, the complex conjugate $c$ acts on $f$ by the complex conjugate of its Fourier coefficients.
    In terms of automorphic forms, we have
    \[
{^c}\varphi_f
\left(\abcd\right)
=
\overline{\varphi_f
\left(
\begin{pmatrix}
    a & -b \\
    -c & d
\end{pmatrix}
\right)},
\]
where $\varphi_f$ is an automorphic form corresponding to $f$.
This yields the assertion.
\end{proof}

For $\sigma \in \Aut(\bbC)$ and a cuspidal automorphic representation $\pi$ with $\varphi \in \pi$, let ${^\sigma\pi_1}$ be the cuspidal automorphic representation generated by ${^\sigma\varphi}$.
By Theorem \ref{thm_int_rep_global_L} and Lemma \ref{lem_cpx_conj}, the restriction of $E(g, s_0'-1/2,f)$ to $G_n(\bbA_F) \times G_n(\bbA_F)$ is of the form
\begin{equation}\label{exp_eisenstein_ser}
\sum_{\xi} \sum_{\varphi \in \calB(\lambda, \frakN)[\xi]} 
\sum_{\varphi' \in \calB(\lambda, \frakN)[^c\xi_1]} 
C(\varphi,\varphi') \cdot \varphi(g) \varphi'(h) 
+
\sum_{\calE,\calE' \in \calB_\Eis(\lambda, \frakN)} C(\calE,\calE') \cdot \calE(g) \calE'(h) 
\end{equation}
for some constants $C(\varphi,\varphi')$ and $C(\calE,\calE')$.
\begin{lem}\label{lem_C(varphi)}
For $\varphi \in \calB(\lambda, \frakN)[\pi]$, we have
\[
\sum_{\varphi' \in \calB(\lambda, \frakN)[^c\pi_1]}C(\varphi, \varphi') \varphi'(h) = C(\varphi)\prod_{v \in S_1 \sqcup S_2} 
        r(\diag(\tau_v \cdot 1_n, \tau_v^{-1}\cdot 1_n))
        \overline{\varphi(\check{h})},
\]
where $C(\varphi)$ is defined by
\begin{align*}
    C(\varphi)&=C \cdot
        \frac{D_F^{-n(2n+1)/2}\pi^{dn(n+1)/2}}
{\prod_{j=1}^{n}\Gamma(j)^d\zeta_F(2j)}  
\cdot
        \frac{L^S(s_0', \pi, \mu)}{b^S_{2n}(s_0'-1/2,\mu)}\cdot \prod_{v \in S_1}\gamma_v(1-s_0', \pi_v^\vee, \mu_v^{-1}, \psi_{v, \bdelta_v^{-1}}) 
        \\
        & \quad \cdot \prod_{v \in S_1}
        \gamma_v(-s_0'-n+1, \mu_v^{-1},\psi_{v,\bdelta_v^{-1}})^{-1}\prod_{j=1}^{n}\gamma_{v}(-2s_0' - 2n + 2j+1, \mu_v^{-2}, \psi_{v,\bdelta_v^{-1}})^{-1}
        \\
        &\quad \cdot \prod_{v \in S_1 \sqcup S_2} \mu_v(\tau_v)^{n} \cdot \prod_{v \mid \infty} (\sqrt{-1}\,\pi)^{-\sum_{j=1}^n \lambda_{v,j} + n(s_0'+n)}
        \pi^{n(n+1)/2} \cdot \langle \varphi, \varphi \rangle^{-1}.
\end{align*}
\end{lem}
\begin{proof}   
    The left hand side of Proposition \ref{prop_global_zeta_int_explicit} is equal to $\langle\varphi, \varphi \rangle \sum_{\varphi' \in \calB(\lambda, \frakN)[^c\pi_1]}C(\varphi, \varphi') \varphi'$, which implies the statement.
\end{proof}

We now prove the main theorem.
Fix a rational structure on the $K_\infty$-type $\lambda = (\lambda_v)_v$.
\begin{thm}\label{thm_main}
    Under the assumptions of \S \ref{subsec_choice_global_section}, let $\varphi$ be a cusp form in $\pi$ corresponding to a Hilbert-Siegel cusp form whose Fourier coefficients lie in a CM field.
    For any $\sigma \in \Aut(\bbC)$, we denote by $\pi_{{^\sigma \varphi}}$ the cuspidal automorphic representation generated by ${^\sigma \varphi}$.
    We then have
    \begin{align*}
    &\sigma\left(
    \frac{L^S(s_0', \pi, \mu)}{D_F^{(n+1)/2}G(\mu)^{n+1}(2\pi \sqrt{-1})^{\sum_{v \mid \infty}{\sum_{j=1}^n \lambda_{v,j} + d(n+1)s_0'}}\langle\varphi, \varphi\rangle}
    \prod_{v \in S_1} \frac{\gamma_v(1-s_0', \pi^\vee_v, \mu_v^{-1}, \psi_{v, \bdelta_v^{-1}})}{\vep_v(0, \mu_v^{-1}, \psi_{v,\bdelta_v^{-1}})^{2n+1}}\right)\\
    &=
    \frac{L^S(s_0', \pi_{{^\sigma \varphi}}, {^\sigma \mu})}{D_F^{(n+1)/2}G({^\sigma \mu})^{n+1}(2\pi \sqrt{-1})^{\sum_{v \mid \infty}{\sum_{j=1}^n \lambda_{v,j} + d(n+1)s_0'}}\langle{^\sigma \varphi}, {^\sigma \varphi}\rangle}
    \prod_{v \in S_1} \frac{\gamma_v(1-s_0', \pi_{{^\sigma \varphi}, v}^\vee, {^\sigma}\mu_v^{-1}, \psi_{v, \bdelta_v^{-1}})}
    {\vep_v(0, {^\sigma\mu_v^{-1}}, \psi_{v,\bdelta_v^{-1}})^{2n+1}}.
\end{align*}
\end{thm}
\begin{proof}
    Let $\calB(\lambda,\frakN)[\xi] = \{\varphi_{\xi,1},\ldots,\varphi_{\xi, \ell_\xi}\}, \calB(^\sigma\lambda,\frakN)[\xi'] = \{\varphi'_{\xi',1},\ldots,\varphi'_{\xi', \ell'_\xi}\}$ be orthogonal bases so that $\varphi_{\pi, 1} = \varphi$ and $\varphi_{^\sigma\pi_1,1} = {^\sigma}\varphi$.
    Applying the description \eqref{exp_eisenstein_ser} to ${^\sigma}E(g,s_0'-1/2,f, \mu, (\lambda_v)_v, (1)_v) = E(g,s_0'-1/2,f, {^\sigma}\mu, (\lambda_
    {\sigma \circ v})_v, (\tau_{\sigma,v}^{-1})_v)$ proved in Proposition \ref{prop_rationality_Eis_ser} with $\tau_v=1$, we have
\begin{align*}
&\sum_{\xi} \sum_{\varphi \in \calB(\lambda, \frakN)[\xi]} 
\sum_{\varphi' \in \calB(\lambda, \frakN)[^c\xi_1]} 
\sigma(C(\varphi,\varphi')) \cdot {^\sigma\varphi}(g) {^\sigma}\varphi'(h) 
+
\sum_{\calE,\calE' \in \calB_\Eis(\lambda, \frakN)} \sigma(C(\calE,\calE')) \cdot {^\sigma}\calE(g) {^\sigma}\calE'(h)\\
&=
\sum_{\xi'} \sum_{\varphi_1 \in \calB(^\sigma\lambda, \frakN)[\xi']} 
\sum_{\varphi'_1 \in \calB(^\sigma\lambda, \frakN)[^c\xi_1']} 
C(\varphi_1,\varphi'_1) \cdot \varphi_1(g) \varphi'_1(h) 
+
\sum_{\calE_1,\calE'_1 \in \calB_\Eis(^\sigma\lambda, \frakN)} C(\calE_1,\calE'_1) \cdot \calE_1(g) \calE'_1(h)
\end{align*}
Take the Petersson inner product with ${^\sigma}\varphi$ in the variable $g$.
By the orthogonality of the basis and Lemmas \ref{lem_cpx_conj} and \ref{lem_C(varphi)}, we obtain
\begin{align}\label{formula_exp_eis_ser_in_h}
&\sum_{\xi} \sum_{\ell = 1}^{\ell_\xi} 
\sum_{m=1}^{\ell_{{^c}\xi_1}} 
\sigma(C(\varphi_{\xi,\ell},\varphi_{^c\xi_1,m})) \langle{^\sigma}\varphi_{\xi, \ell}, {^\sigma}\varphi\rangle
{^\sigma}\varphi_{^{c}\xi_1,m}(h) 
+
\sum_{\calE,\calE' \in \calB_\Eis(\lambda, \frakN)} \sigma(C(\calE,\calE'))  \langle{^\sigma}\calE, {^\sigma}\varphi\rangle {^\sigma}\calE'(h)\\
&=
\langle {^\sigma}\varphi,{^\sigma}\varphi \rangle
C(^\sigma\varphi)
\cdot
r
\left(
\begin{pmatrix}
    0 & -\tau_{\sigma,v} \cdot 1_n\\
    \tau_{\sigma,v}^{-1} \cdot 1_n & 0
\end{pmatrix}_v
\right)
{^{c \circ \sigma}}\varphi(h).\notag
\end{align}
Since the Fourier coefficients of the modular form corresponding to $\varphi$ lie in a CM field, one has ${^{c \circ \sigma}\varphi} = {^{\sigma \circ c}\varphi}$.
For automorphic forms $F$ and $G$, we have $\overline{\langle {^c F}, {^c G}\rangle} = \langle F,G \rangle$.
Hence, all automorphic forms in ${^c}\calB_{\Eis}(\lambda, \frakN)$ are orthogonal to all cusp forms, and the union $\bigcup_\xi\calB(\lambda, \frakN)[\xi] \cup {^c}\calB_{\Eis}(\lambda, \frakN)$ is linearly independent.
Every right translation of an element in $\calB_{\Eis}(\lambda, \frakN)$ is orthogonal to cusp forms.
By the linear independence, all terms on the left hand side of \eqref{formula_exp_eis_ser_in_h} vanish except for $\xi = \pi$.
Applying Lemma \ref{lem_C(varphi)} to the left hand side, we obtain
\begin{align*}
&\sum_{\ell = 1}^{\ell_\pi} 
\langle{^\sigma}\varphi_{\pi, \ell}, {^\sigma}\varphi\rangle 
\sigma(C(\varphi_{\pi,\ell}))
r
\left(
\begin{pmatrix}
    0_n & -\tau_{\sigma,v} \cdot 1_n\\
    \tau_{\sigma_v} \cdot 1_n& 0_n
\end{pmatrix}_v\right)
{^{\sigma \circ c}}\varphi_{\pi,\ell}(h)\\
&\qquad=
\langle {^\sigma}\varphi,{^\sigma}\varphi \rangle
C(^\sigma\varphi)
\cdot
r
\left(
\begin{pmatrix}
    0 & -\tau_{\sigma,v} \cdot 1_n\\
    \tau_{\sigma,v}^{-1} \cdot 1_n & 0
\end{pmatrix}_v
\right)
{^{\sigma \circ c}}\varphi(h).
\end{align*}
Since $\varphi_{\pi,1} = \varphi$, the linear independence again gives
\[
\sigma(C(\varphi)) = C(^\sigma \varphi).
 \]
This equality yields the desired equality by Lemma \ref{lem_arith_Gauss_sum_L-val}.
\end{proof}

The following corollary gives arithmeticity in the right half plane.
This result is essentially contained in \cite{2020_Liu} and \cite[Theorem 5.13]{22_Shih-yu_Chen}.

\begin{cor}\label{cor_S_1_empty}
    Let $\pi$ be a cuspidal automorphic representation of $G_n(\bbA_F)$ and $\mu$ be a Hecke character of $F^\times F_{\infty, +}^\times \bs \bbA_F^\times$.
    Let $S$ be a finite set of places containing the ramified places and the archimedean places.
    Suppose that for all archimedean places $v$, $\pi_v$ is a holomorphic discrete series representation of weight $\lambda_v = (\lambda_{v,1},\ldots,\lambda_{v,n})$, and $L(s,\pi,\mu)$ has a critical point $m>0$.
    Take a cusp form $\varphi$ in $\pi$ corresponding to a Hilbert-Siegel cusp form whose Fourier coefficients lie in a CM field.
    For any $\sigma \in \Aut(\bbC)$, we denote by $\pi_{{^\sigma \varphi}}$ the cuspidal automorphic representation generated by ${^\sigma \varphi}$.
    We then have
        \begin{align*}
    &\sigma\left(
    \frac{L^S(m, \pi, \mu)}{D_F^{(n+1)/2}G(\mu)^{n+1}(2\pi \sqrt{-1})^{\sum_{v \mid \infty}{\sum_{j=1}^n \lambda_{v,j} + d(n+1)m}}\langle\varphi, \varphi\rangle}\right)
    \\
    &\qquad
    =
    \frac{L^S(m, \pi_{{^\sigma \varphi}}, {^\sigma \mu})}{D_F^{(n+1)/2}G({^\sigma \mu})^{n+1}(2\pi \sqrt{-1})^{\sum_{v \mid \infty}{\sum_{j=1}^n \lambda_{v,j} + d(n+1)m}}\langle{^\sigma \varphi}, {^\sigma \varphi}\rangle},
\end{align*}
    except for the case where $m=1, F =\bbQ$ and $\mu^2 =\mathbf{1}$.
\end{cor}
\begin{proof}
    This is exactly the case of Theorem \ref{thm_main} with $s_0 = m > 0$ and $S_1 = \emptyset$.
\end{proof}

To establish the arithmeticity in the left half plane, we recall the functional equation of $L(s,\pi,\mu)$.

\begin{thm}[{\cite[Theorem 4]{2005_Lapid-Rallis}}]\label{thm_FE_std_L}
    With the notation in \S \ref{subsec_choice_global_section}, let $S'$ be a finite set of places containing all the archimedean places and the places where either $\pi, \mu$, or $\psi$ is ramified.
    Then,
    \[
    L^{S'}(s, \pi,\mu) = \prod_{v \in S'} \gamma_v(s, \pi_v, \mu_v, \psi_v) \cdot L^{S'}(1-s, \pi^\vee, \mu^{-1}).
    \]
\end{thm}

Let $S'$ be a finite set of places containing $S$ such that $\pi, \mu$, and $\psi$ are unramified for any $v \not\in S'$.
We may rewrite the functional equation in Theorem \ref{thm_FE_std_L} in the following convenient form
\[
L^{S}(s, \pi^\vee,\mu^{-1}) = \prod_{v \in S} \gamma_v(s, \pi_v^\vee, \mu_v^{-1}, \psi_{v,\bdelta_v^{-1}}) 
\cdot
\prod_{v \in S' \setminus S_\infty} 
(\mu_v(\bdelta_v)|\bdelta_v|_v^{1/2-s})^{-(2n+1)}
\cdot L^{S}(1-s, \pi, \mu)
\]
where $\bdelta_v = 1$ for archimedean places $v$.

\begin{cor}\label{main_thm_s_0_negative}
    With the notation of Corollary \ref{cor_S_1_empty}, if $m$ is nonpositive, we have
    \begin{align*}
    &\sigma\left(\frac{L^S(m, \pi, \mu)}{D_F^{n/2}G(\mu)^n(2\pi\sqrt{-1})^{\sum_{v \mid \infty}\sum_{j=1}^n\lambda_{v,j} + dnm}\langle\varphi, \varphi\rangle}\right)\\
    &\qquad =
    \frac{L^S(m, \pi_{{^\sigma \varphi}}, {^\sigma \mu})}{D_F^{n/2}G(^\sigma\mu)^n(2\pi\sqrt{-1})^{\sum_{v \mid \infty}\sum_{j=1}^n\lambda_{v,j} + dnm}\langle{^\sigma \varphi}, {^\sigma \varphi}\rangle}
    \end{align*}
    except for the case where $m=0,F=\bbQ, S = S_\infty$ and $\mu = \mathbf{1}$.
\end{cor}
\begin{proof}
    For an archimedean place $v$, we have $\gamma_v(1-s_0', \pi_v^\vee, \mu_v^{-1}, \psi_{v})^{-1} \in (\sqrt{-1}\,\pi)^{(2n+1)s_0'-n} \bbQ^\times$.
    Take $S_2 = \emptyset$ in Theorem \ref{thm_main}. 
    By replacing $\mu$ by $\mu^{-1}$ if necessary, we obtain the desired Galois equivariance using \eqref{selfduality_local_factor} except for the case $m=0, F=\bbQ, S = S_\infty$ and $\mu^2 = \mathbf{1}$.
    This exceptional case occurs only when $\mu = \mathbf{1}$, which is excluded.
\end{proof}

\begin{rem}\label{rem_s=0_except}
    The only exception to the arithmeticity for $m \leq 0$ is $m = 0, F=\bbQ,S=S_\infty$ and $\mu=\mathbf{1}$.
    This is exactly the case where $\pi$ is a cuspidal automorphic representation generated by a full level Siegel cusp form and $\mu = \mathbf{1}$.
    The corresponding $L$-function is the standard $L$-function $L^{S_\infty}(s,\pi)$, for which arithmeticity at $s=1$ was proved in \cite{91_Mizumoto, 00_Kozima} for the case $n \equiv 3 \bmod 4$.
    The functional equation yields the arithmeticity at $s=0$.
    The remaining case is $n \equiv 1 \bmod 4$, whose arithmeticity remains open when $n \neq 1$.
    Note that the corresponding Eisenstein series is not nearly holomorphic.
    See \cite[\S 17]{00_Shimura} and \cite{26_Takeda_Eis_ser} for details on the Eisenstein series.
\end{rem}

For the remaining case where $s_0=1,F= \bbQ$ and $\mu^2 = \mathbf{1}$, we establish the arithmeticity under an additional condition on the local components.

\begin{thm}\label{thm_s_0=1}
    With the notation of Corollary \ref{cor_S_1_empty}, suppose that $\mu^2 = \mathbf{1}$ and $m=1$ is a critical point.
    If there exists a nonarchimedean place $v$ such that $\pi_v$ is unramified and $\mu_v$ is ramified, then 
    \begin{align*}
    &\sigma\left(
    \frac{L^S(1, \pi, \mu)}{D_F^{(n+1)/2}G(\mu)^{n+1}(2\pi \sqrt{-1})^{\sum_{v \mid \infty}{\sum_{j=1}^n \lambda_{v,j} + d(n+1)}}\langle\varphi, \varphi\rangle}\right)
    \\
    &\qquad
    =
    \frac{L^S(1, \pi_{{^\sigma \varphi}}, {^\sigma \mu})}{D_F^{(n+1)/2}G({^\sigma \mu})^{n+1}(2\pi \sqrt{-1})^{\sum_{v \mid \infty}{\sum_{j=1}^n \lambda_{v,j} + d(n+1)}}\langle{^\sigma \varphi}, {^\sigma \varphi}\rangle}.
\end{align*}
\end{thm}
\begin{proof}
    Let $v_0$ be a place satisfying the assumption.
    The definition of the $\vep$-factor using the local Langlands correspondence shows that
    \begin{equation}\label{eqn_obstruction_s=1}
    \frac{\gamma_{v_0}(0, \pi^\vee_{v_0}, \mu_{v_0}^{-1}, \psi_{v_0, \bdelta_{v_0}^{-1}})}{\vep_{v_0}(0, \mu_{v_0}^{-1}, \psi_{v_0,\bdelta_{v_0}^{-1}})^{2n+1}}
    \end{equation}
    is $1$.
    Taking $S_1 = \{v_0\}$ and $S_2 = S \setminus (S_\infty \sqcup \{v_0\})$, we obtain the statement.
\end{proof}

To conclude the paper, we give a remark on the assumption of Theorem \ref{thm_s_0=1}.
\begin{rem}\label{rem_s=1}
The assumption that there exists a nonarchimedean place $v$ such that $\pi_v$ is unramified and $\mu_v$ is ramified can be replaced by the assumptions that  \eqref{eqn_obstruction_s=1} is finite, nonzero, and Galois equivariant.
The Galois equivariance of \eqref{eqn_obstruction_s=1} would follow formally from Galois actions on the zeta integrals and intertwining operators (cf.~\cite[Lemma 4.6]{90_Clozel}).
The main difficulty is to characterize the representation $\pi_v$ for which the expression \eqref{eqn_obstruction_s=1} is finite and nonzero.
\end{rem}

\bibliographystyle{alpha}
\bibliography{ref}

\end{document}